\documentclass{article}
\usepackage{microtype}
\usepackage{graphicx}
\usepackage{subfigure}
\usepackage{booktabs} 

\usepackage{color}
\usepackage{hyperref}
\usepackage{amsmath}
\usepackage{amsthm}
\usepackage{amssymb}

\usepackage{multirow}
\usepackage{diagbox}

\newtheorem{thr}{Theorem} 	
\newtheorem{lem}{Lemma}  
\newtheorem{prop}{Proposition} 	
\newtheorem{cor}{Corollary}
\newtheorem{asup}{Assumption}  
\newtheorem{define}{Definition}  
\newtheorem{rmk}{Remark} 

\DeclareMathOperator*{\argmin}{argmin}

\usepackage[accepted]{icml2019}

\icmltitlerunning{Efficient Algorithm for Solving Constrained Lasso}

\begin{document}

\twocolumn[
\icmltitle{An Efficient Augmented Lagrangian Based Method for Constrained Lasso}



\icmlsetsymbol{equal}{*}

\begin{icmlauthorlist}
\icmlauthor{Zengde Deng}{ed}
\icmlauthor{Anthony Man-Cho So}{ed}
\end{icmlauthorlist}

\icmlaffiliation{ed}{Department of Systems Engineering and Engineering Management, the Chinese University of Hong Kong}

\icmlcorrespondingauthor{Zengde Deng}{zddeng@se.cuhk.edu.hk}

\icmlkeywords{Machine Learning, ICML}

\vskip 0.3in
]



\printAffiliationsAndNotice{}  


\begin{abstract}
	Variable selection is one of the most important tasks in statistics and machine learning. 
	To incorporate more  prior information about the regression coefficients, the constrained Lasso model has been proposed in the literature. 
	In this paper, we present an inexact augmented Lagrangian method to solve the Lasso problem with linear equality constraints. By fully exploiting second-order sparsity of the problem, we are able to greatly reduce the computational cost and obtain highly efficient implementations.
	Furthermore, numerical results on both synthetic data and real data show that our algorithm is superior to existing first-order methods in terms of both running time and solution accuracy.
\end{abstract}

\section{Introduction}
With the advent of big data era, variable selection has received great attention in statistics and machine learning since contemporary applications often involve a large number of variables.
There exist a host of methods to address this problem, such as Lasso \cite{tibshirani1996regression}, SCAD \cite{fan2001variable}, elastic net \cite{zou2005regularization}, adaptive Lasso \cite{zou2006adaptive}, relaxed Lasso \cite{meinshausen2007relaxed} and so on.
Benefiting from the simple formulation and the powerful modeling concerning  the variable selection task, Lasso has been extensively applied in various instances \cite{burnham2003model,candes2008introduction,chen2001atomic}. 

In spite of the overwhelming success of Lasso, it still suffers from the limited information induced by $l_1$ norm. To circumvent these issues, researchers have proposed the constrained Lasso model \cite{gaines2018algorithms,james2013penalized} to incorporate more prior information. In the view of above discussions, our goal in this paper is to propose an efficient algorithm to tackle the following constrained Lasso problem
\begin{align} \label{1.1}
\begin{split}
\min_{x}\ &\dfrac{1}{2}\|Ax-b\|^2+\lambda\|x\|_{1} \\
\text{s.t.}\ &Bx=d,
\end{split}
\end{align}
where $b\in\mathbb{R}^m$ is the response vector, $A\in\mathbb{R}^{m\times n}$ is the design matrix of covariates, $x\in\mathbb{R}^{n}$ is the regression estimator, and $B\in\mathbb{R}^{s\times n},d\in\mathbb{R}^{s}$ are given constraints. 

An important example which falls into the constrained Lasso problem is Lasso with sum-to-zero constraints, i.e. $e^{T}x=0$. This type of constraint has been adopted in microbiome data regression \cite{shi2016regression} and variable selection \cite{lin2014variable} where the covariates come from compositional data. 

Another example which is widely used in statistics is the generalized Lasso problem 
\begin{align}\label{1.3}
\min_{x} \frac{1}{2}\|Ax-b\|^2+\|Dx\|_1,
\end{align}
where $D\in\mathbb{R}^{p\times n}$. When $\text{rank}(D)=p$ and $p\leq n$, Tibshirani \yrcite{tibshirani2011solution} has derived that (\ref{1.3}) can be transformed into a Lasso problem. In fact, (\ref{1.3}) is a special case of constrained Lasso with $d=0$  \cite{gaines2018algorithms,james2013penalized} when $p\geq n$ and $D$ has full column rank $n$ and we elaborate on this in section \ref{sec6}.

\subparagraph{Our contributions}
\indent 
In this paper, we propose a semismooth Newton augmented Lagrangian method to solve the constrained Lasso problem. To fully exploit the sparsity structure, we mainly focus on the dual formulation of our problem and an inexact augmented Lagrangian method is proposed.
The main challenge lies in how to solve the subproblem of augmented Lagrangian method efficiently. 
To overcome this difficulty, we apply the semismooth Newton method to resolve the inner subproblem. 
As we will show in numerical experiments, we only need about tens or dozens of outer iterations and about ten iterations to solve each subproblem to the desired accuracy and the total running time is small. The key insights behind this impressive performance are three-fold: (a) Regarding the outer loop, we have superlinear convergence to achieve highly accurate solution; (b) Besides, we also attain superlinear convergence in the inner subproblem solver and hence the total iteration number of our algorithm is still small; (c) For each iteration in solving the inner subproblem, we extensively exploit the second-order sparsity of the problem to greatly reduce the computational cost. In summary, not only can we prove the theoretical effectiveness of our algorithm, but also provide highly efficient implementations by exploring the hidden structure of our problem.


\section{Related work}
\indent The constrained Lasso problem can be seen as a special case of the composite optimization problem with linear equality constraints. Without linear constraints, the proximal gradient method can achieve $\mathcal{O}(1/k)$ convergence rate and accelerated proximal gradient \cite{beck2009fast,nesterov2013gradient} obtain $\mathcal{O}(1/k^2)$ rate which is optimal for first-order methods to solve composite convex optimization problems. Although the methods proposed in Beck et al.\yrcite{beck2009fast}, Nesterov \yrcite{nesterov2013gradient} can also handle constrained problems, they require projections onto the constraint set, which can be computational prohibitive. Fortunately, there are other choices to tackle structured linear constrained problems. One of the most popular methods is the augmented Lagrangian method (ALM) \cite{bertsekas2014constrained}. Due to the difficulty of solving the subproblem, linearization technique is utilized in Yang et al. \cite{yang2013linearized}, which results in linearized ALM (LALM). The recent work Xu \yrcite{xu2017accelerated} proposes an accelerated linearized ALM to solve linearly constrained composite convex problem. Besides, alternating direction method of multipliers (ADMM) \cite{boyd2011distributed} is applied to solve the linear constrained problem and Goldstein et al. \yrcite{goldstein2014fast} propose a faster version of ADMM (A-ADMM). Recently, first-order primal-dual methods \cite{chambolle2011first} are also utilized to handle this kind of problems. Unfortunately, the acceleration rates of above algorithms all need the objective function to be strongly convex which is not satisfied by our problem.

\indent More recently, semismooth Newton augmented Lagrangian method is attracting more and more attention due to its fast convergence and good experimental performance. This kind of method has been used to tackle problems such as SDP \cite{zhao2010newton}, quadratic SDP \cite{li2015qsdpnal}, Lasso \cite{li2018highly} and fused Lasso \cite{li2018efficiently}, convex clustering \cite{pmlr-v80-yuan18a} and so on. Inspired by this, we propose a new ALM based framework to cope with constrained Lasso problem in this paper.

\section{Problem formulation and an augmented Lagrangian method} \label{sec3}
In this section, we propose an augmented Lagrangian metthod to solve problem (\ref{1.1}). Before that, we introduce some properties of our problem. 
\subsection{Dual problem and optimality conditions}
In this subsection we first consider the general case of our problem (\ref{P}) and derive the dual problem (\ref{D}). Moreover, we give the optimality conditions associated with (\ref{D}) which will play an important role in the convergence analysis in section \ref{sec5}.

Our constrained Lasso problem (\ref{1.3}) can be seen as a special case of following problem:
\begin{equation}\label{P}
\begin{split}
\min_{x}\ &\{f(x)=h(Ax)+p(x)\} \\
\text{s.t.}\ &Bx=d,
\end{split} \tag{P}
\end{equation}
where $A\in\mathbb{R}^{m\times n}, x\in\mathbb{R}^n, B\in\mathbb{R}^{s\times n},d\in\mathbb{R}^{s}$, $h:\mathcal{U}\rightarrow(-\infty,+\infty]$ is continuous differentiable on $\text{dom}(h)$ and $p:\mathcal{X}\rightarrow(-\infty,+\infty]$ is a closed proper convex function. The dual problem of (\ref{P}) can be written as
\begin{equation}\label{D}
\begin{split}
\min_{u,v,w}\ &\{g(z):= h^{*}(u)-\langle v,d\rangle+p^{*}(w)\} \\
\text{s.t.}\ &A^{T}u-B^{T}v+w=0,
\end{split}\tag{D}
\end{equation}
where $h^{*}$ and $p^{*}$ are conjugate functions of $h$ and $g$ respectively and we denote $z:=(u,v,w)\in\mathcal{Z}:=\mathcal{U}\times\mathcal{V}\times\mathcal{W}$ with $u\in\mathbb{R}^{m},v\in\mathbb{R}^{s}$ and $w\in\mathbb{R}^{n}$. In our constrained Lasso model, we have  $h(x)=\frac{1}{2}\|x-b\|^2$ and $p(x)=\lambda\|x\|_1$. Therefore $h^{*}(u)=\frac{1}{2}\|u\|^2+b^{T}u$ and $p^{*}(w)=\mathbb{I}_{\{\|w\|_{\infty}\leq\lambda\}}$ where $\mathbb{I}$ is the indicator function. 

Denote the Lagrangian function associated with (\ref{D}) by $l$ for $\forall(z,x)\in\mathcal{Z}\times\mathcal{X}$, i.e.,
\begin{align}
\begin{split}
\ &l(z;x)\equiv l(u,v,w;x)\\
:=\ &h^{*}(u)-\langle v,d\rangle+p^{*}(w)-\langle x,A^{T}u-B^{T}v+w\rangle.
\end{split}
\end{align}
For the convex-concave function $l$, we further define 
\begin{align}
\begin{split}
\ &\psi(z):=\sup_{x\in\mathcal{X}}l(z;x),\quad\forall z\in\mathcal{Z}, \\ \ &\phi(x):=\inf_{z\in\mathcal{Z}}l(z;x),\quad\forall x\in\mathcal{X}.
\end{split}
\end{align}
Moreover, we define the following mappings corresponding to $l,\psi$ and $\phi$:
\begin{align*}
\ &\mathcal{T}_{l}(z,x):=\{(\hat{z},\hat{x})\in\mathcal{Z}\times\mathcal{X}|(\hat{z},-\hat{x})\in\partial l(z,x)\}, \\
\ &\mathcal{T}_{\psi}(z):=\partial\psi(z),\ \forall z\in\mathcal{Z},\quad\mathcal{T}_{\phi}(x):=\partial\phi(x),\ \forall x\in\mathcal{X}.
\end{align*}
The KKT systems associated with (D) is given by
\begin{align}\label{2.6}
\begin{split}
\begin{cases}
0\in\partial h^{*}(u)-Ax, \\
0\in\partial p^{*}(w)-x, \\
0 = Bx-d, \\
0= A^Tu-B^{T}v+w,
\end{cases} \quad(u,v,w;x)\in\mathcal{Z}\times\mathcal{X}.
\end{split}
\end{align}
Suppose that the KKT system (\ref{2.6}) has at least one solution. Let $\bar{z}:=(\bar{u},\bar{v},\bar{w})$ be an optimal solution to (\ref{D}). Denote $\bar{x}\in\mathcal{M}_{\psi}(\bar{z})$, where $\mathcal{M}_{\psi}(\bar{z})$ is the set of Lagrangian multipliers associated with $\bar{z}$, then $(\bar{z},\bar{x})$ is a solution of KKT system (\ref{2.6}). By the property of conjugate function, we see that $(\bar{x},\bar{v})$ solves the following KKT system: 
\begin{equation}
\resizebox{.9\hsize}{!}{$
	\begin{aligned}
	\begin{cases}
	0\in A^{T}\nabla h(Ax)-B^{T}v+\partial p(x), \\
	0=Bx-d,
	\end{cases} \ (x,v)\in\mathcal{X}\times\mathcal{V}.
	\end{aligned}$}
\end{equation} 
Conversely, if we set $\bar{u}=\nabla h(A\bar{x})$ and $\bar{w}=-A^T\bar{u}+B^T\bar{v}$, then $(\bar{u},\bar{v},\bar{w},\bar{x})$ solves the KKT system (\ref{2.6}). From Rockafellar \yrcite{rockafellar2015convex}, we know that $(\bar{z},\bar{x})$ solves KKT system (\ref{2.6}) if and only if $\bar{x}$ solves (\ref{P}) and $\bar{z}$ solves (\ref{D}).

\subsection{An augmented Lagrangian method} \label{sec3.2}
In this section, we introduce the inexact augmented Lagrangian method to solve (\ref{D}). The augmented Lagrangian function associated with (\ref{D}) is given by
\begin{align*}
\mathcal{L}_{\sigma}(u,v,w;x)=l(u,v,w;x)+\dfrac{\sigma}{2}\|A^{T}u-B^{T}v+w\|^2.
\end{align*}
We propose the following inexact augmented Lagrangian method for (\ref{D}).
\begin{algorithm} 
	\caption{Inexact ALM for (D)} 
	\label{alg1} 
	\begin{algorithmic}[1] 
		\STATE {\bfseries Input:} $u^{0},v^{0},w^{0},x^{0}$.
		\FOR{$k=0,1,\dots$}
		\STATE Get an approximate solution 
		\begin{equation}\label{3.1}
		\resizebox{.77\hsize}{!}{$
			\begin{aligned}	
			z^{k+1}\ &=(u^{k+1},v^{k+1},w^{k+1}) \\
			\ &\approx\argmin_{z:=(u,v,w)}\left\{\Theta_{k}(z):=\mathcal{L}_{\sigma_{k}}(u,v,w;x^{k})\right\}
			\end{aligned}$}
		\end{equation}
		\STATE Update $x$ by $x^{k+1}=x^{k}-\sigma_{k}(A^{T}u^{k+1}-B^{T}v^{k+1}+w^{k+1})$ and update $\sigma_{k+1}\uparrow\sigma_{\infty}\leq\infty$.
		\ENDFOR		
	\end{algorithmic} 
\end{algorithm}

To obtain an efficient implementation of Algorithm \ref{alg1}, we need to solve the subproblem (\ref{3.1}) inexactly. From Rockafellar's work \yrcite{rockafellar1976augmented}, we use one of the following stopping criteria: 
\begin{align*}
\ &\tag{A}\quad\Theta_{k}(z^{k+1})-\inf\Theta_{k}\leq\varepsilon_{k}^{2}/2\sigma_{k},\label{A} \\
\ &\tag{B}\quad\Theta_{k}(z^{k+1})-\inf\Theta_{k}\leq\zeta_k^2\|x^{k+1}-x^k\|^2/2\sigma_{k}, \label{B}
\end{align*}
where $\sum_{k=0}^{\infty}\varepsilon_{k}<\infty$ and  $\sum_{k=0}^{\infty}\zeta_k<\infty$. 

\section{Inexact semismooth Newton method to solve ALM subproblem}
In section \ref{sec3.2}, we design an inexact ALM to solve (\ref{D}), but the main challenge lies in how to solve the subproblem (\ref{3.1}) efficiently. In this section we propose a semismooth Newton method to solve ALM subproblem (\ref{3.1}) and provide highly efficient implementations by exploiting second-order sparsity in the subproblem.
 
Here, we define the proximal mapping $\text{Prox}_p(\cdot) $ associated with $p$ as
\begin{align*}
\text{Prox}_{p}(x):=\argmin_{u}\left\{p(u)+\dfrac{1}{2}\|u-x\|^2\right\}.
\end{align*} 
Moreover, by the Moreau decomposition, we have $x = \text{Prox}_{tp}(x)+t\text{Prox}_{t^{-1}p^*}(x/t)$ for $t>0$.

Note that the augmented Lagrangian can be expressed as 
\begin{align*}
\ &\mathcal{L}_{\sigma}(u,v,w;x)\\
=\ &h^{*}(u)-\langle v,d\rangle+p^{*}(w)-\langle x,A^{T}u-B^{T}v+w\rangle\\
\ &+\dfrac{\sigma}{2}\|A^{T}u-B^{T}v+w\|^2.
\end{align*}
For a fixed $x$ and given $\sigma$, let $y=(u,v)\in\mathbb{R}^{m+s}$, we consider 
\begin{align}
\min_{y,w}\Theta(y,w)=\mathcal{L}_{\sigma}(y,w;x),
\end{align}
for convenience we set $\bar{h}^{*}(y)=h^{*}(u)-\langle v,d\rangle$.
Then we denote $\theta(y)$ by
\begin{align*}
\theta(y)= \ &\inf_{w}\mathcal{L}_{\sigma}(y,w;x) \\
=\ &\bar{h}^{*}(y)+p^{*}\left(\text{Prox}_{p^{*}/\sigma}(x/\sigma-\bar{A}^{T}y)\right)\\
\ &+\dfrac{1}{2\sigma}\left\|\text{Prox}_{\sigma p}(x-\sigma(\bar{A}^{T}y)\right\|^2-\dfrac{1}{2\sigma}\|x\|^2,
\end{align*}
where $\bar{A}=\begin{bmatrix}
A \\ -B
\end{bmatrix}$, and we can also write $\theta(y)$ as $\theta(u,v)$. Hence, if we let $(\tilde{y},\tilde{w})=\argmin\Theta(y,w)$, then  $(\tilde{y},\tilde{w})$ can be computed in the following manner:
\begin{align} \label{4.2}
\begin{split}
\begin{cases}
\tilde{y}=\argmin \theta(y), \\
\tilde{w}=\text{Prox}_{p^{*}/\sigma}(x/\sigma-\bar{A}^{T}\bar{y}).
\end{cases}
\end{split}
\end{align}
Note that $\theta(\cdot)$ is a continuous differentiable function with
\begin{align*}
\nabla\theta(y)=\begin{bmatrix}
\nabla h^{*}(u)-A\text{Prox}_{\sigma p}(x-\sigma(\bar{A}^{T}y)) \\
-d+B\text{Prox}_{\sigma p}(x-\sigma(\bar{A}^{T}y))
\end{bmatrix}.
\end{align*}
Moreover,  (\ref{4.2}) is equivalent to the following:
\begin{align}\label{4.3}
\nabla\theta(y)=0,\qquad \forall y\in\text{dom}(y).
\end{align}
For any $y\in\text{dom}(y)$, we define
\begin{align*}
\hat{\partial}^2\theta(y):=\ &\begin{bmatrix}
\nabla^{2}h^{*}(u) &  \\
&  \textbf{0}
\end{bmatrix}+\sigma\bar{A}\partial\text{Prox}_{\sigma p}(x-\sigma(\bar{A}^{T}y))\bar{A}^{T}\\
= \ &H+\sigma\bar{A}\partial\text{Prox}_{\sigma p}(x-\sigma(\bar{A}^{T}y))\bar{A}^{T},
\end{align*}
where $\partial\text{Prox}_{\sigma p}(x-\sigma(\bar{A}^{T}y))$ is the Clarke subdifferential \cite{clarke1990optimization} of  $\text{Prox}_{\sigma p}(\cdot)$ at $x-\sigma(\bar{A}^{T}y)$. From Hiriart-Urruty et al., \yrcite{hiriart1984generalized}, we know that
\begin{equation}\label{4.4}
\resizebox{.87\hsize}{!}{$
	\hat{\partial}^2\theta(y)(d_u,d_v)=\partial^2\theta(y)(d_u,d_v),\ \forall (d_u,d_v)\in\text{dom}(y)$},
\end{equation}
where $\partial^2\theta(y)$ denotes the generalized Hessian of $\theta(\cdot)$ at $y$. 
Define 
\begin{align}\label{4.5}
V:= H+\sigma \bar{A}Q\bar{A}^{T}
\end{align}
with $Q\in\partial\text{Prox}_{\sigma p}(x-\sigma(\bar{A}^{T}y))$,
then we have $V\in\hat{\partial}^2\theta(y)$.  Since both $H$ and $\bar{A}Q\bar{A}^{T}$ are positive semidefinite, $V$ is positive semidefinite. Now, let us introduce the notion of semismoothness \cite{mifflin1977semismooth,sun2002semismooth}. 
\begin{define}(Semismoothness)
	Suppose that $F:\mathcal{X}\rightarrow\mathcal{Y}$ is a locally Lipschitz continuous function. $F$ is said to be semismooth at $x\in\mathcal{X}$ if $F$ is directionally differentiable at $x$ and for any $V\in\partial F(x+\Delta x)$ and $\Delta x\rightarrow 0$,
	\begin{align*}
	F(x+\Delta x)-F(x)-V(\Delta x)=o(\|\Delta x\|).
	\end{align*}
	$F$ is said to be strongly semismooth  at $x\in\mathcal{X}$ if 
	\begin{align*}
	F(x+\Delta x)-F(x)-V(\Delta x)=O(\|\Delta x\|^2).
	\end{align*}
	$F$ is said to be a semismooth (strongly semismooth) function on $\mathcal{X}$ if it is semismooth (strongly semismooth) everywhere in $\mathcal{X}$.
\end{define}
Note that all twice continuous differentiable functions and piecewise linear functions are strongly semismooth everywhere \cite{li2018highly}. In our constrained Lasso problem, $h^{*}(\cdot)$ is twice continuous differentiable and $\text{Prox}_{\lambda\|x\|_1}(\cdot)$ is piecewise linear. Hence, they are all strongly semismooth.

Now, we are ready to give an efficient inexact semismooth Newton (SSN) method to solve equation (\ref{4.2}) in Algorithm \ref{alg2}. 

\begin{algorithm} [tb] 
	\caption{Semismooth Newton (SSN) for subproblem} 
	\label{alg2} 
	\begin{algorithmic}[1] 
		\STATE {\bfseries Input:} Given $\mu\in(0,1/2)$, $\bar{\eta}\in(0,1)$,  $\tau\in(0,1],\tau_1,\tau_2\in(0,1)$ and $\delta\in(0,1)$. Choose $y^{0}=(u^{0},v^{0})$.
		\FOR{$j=0,1,\dots$}
		\STATE Choose $Q^{j}\in\partial\text{Prox}_{\sigma p}(x-\sigma(\bar{A}^{T}y^j))$. Let $V_{j}$ be given as in (\ref{4.5}) and $\epsilon_{j}=\tau_1\min\{\tau_2,\|\nabla\theta(y^{j})\|\}$. Solve the following linear system
		\begin{align}\label{4.6}
		V_j(d_u,d_v)+\epsilon_{j}(0,d_v)=-\nabla\theta(y^j)
		\end{align}
		exactly or by CG algorithm to find an approximate solution such that
		\begin{align*}
		\ &\|V_j(d_u^j,d_v^j)+\epsilon_{j}(0,d_v^j)+\nabla\theta(y^j)\|\\
		\leq\ &\min(\bar{\eta},\|\nabla\theta(y^{j}\|^{1+\tau}).
		\end{align*}
		\STATE (Line search) Set $\alpha_j=\delta^{l_{j}}$, where $l_{j}$ is the first nonnegative integer $l$ for which 
		\begin{align*}
		\ &\theta(u^{j}+\delta^{l}d_u^{j},v^{j}+\delta^{l}d_v^{j})\\
		\leq\ &\theta(u^{j},v^j)+\mu\delta^{l}\langle\nabla\theta(y^j),(d_u^j,d_v^j)\rangle.
		\end{align*}
		\STATE Set $u^{j+1}=u^j+\alpha_jd_u^j$ and $v^{j+1}=v^j+\alpha_jd_v^j$
		\ENDFOR
	\end{algorithmic} 
\end{algorithm}


\subsection{Efficient implementation of SSN}
As mentioned before, the key step to determine the efficiency of the whole algorithmic framework is how to solve (\ref{3.1}) quickly. 
We use the semismooth Newton method to tackle (\ref{3.1}) and the main computational cost lies in (\ref{4.6}), which computes the inexact Newton direction. 
Thus we will give efficient implementations to compute (\ref{4.6}).

\indent Recall the definition of $V$ in (\ref{4.5}), we write (\ref{4.6}) as
\begin{align}\label{4.7}
(H_{\epsilon}+\sigma\bar{A}Q\bar{A}^{T})d_y=-\nabla\theta(y)
\end{align}
with $H_{\epsilon}$ is defined by
$
H_{\epsilon}:=\begin{bmatrix}
\textbf{I}_{m} & \\
& \epsilon\textbf{I}_s
\end{bmatrix} $, where $\textbf{I}_m$ denote the $m\times m$ identity matrix. Since $H_{\epsilon}$ is positive definite, our linear system is well defined. 

\indent Before solving (\ref{4.7}), we do Cholesky decomposition on $H_\epsilon$ via $H_\epsilon=LL^{T}$ with $L$ being a lower triangular matrix. Consequently, we rewrite (\ref{4.7}) as 
\begin{equation}\label{4.8}
\resizebox{.85\hsize}{!}{$
	(\textbf{I}_{m+s}+\sigma(L^{-1}\bar{A})Q(L^{-1}\bar{A})^T)(Ld_y)=-L^{-1}\nabla\theta(y)$}.
\end{equation} 
Fortunately, for our constrained Lasso problem we have $h^{*}(u)=\frac{1}{2}\|u\|^2+b^{T}u$. Then $\nabla h^{*}(u)=u+b$, $\nabla^2h^{*}(u)=\textbf{I}_{m}$ and it is easy to check that $L$ can be directly computed as $L=L^{T}=\begin{bmatrix}
\textbf{I}_{m} & \\
& \sqrt{\epsilon}\textbf{I}_s
\end{bmatrix}$ and $L^{-1}=\begin{bmatrix}
\textbf{I}_{m} & \\
& \sqrt{1/\epsilon}\textbf{I}_s
\end{bmatrix}$. Thus, the cost of computing Cholesky decomposition of $H_{\epsilon}$ is negligible. For notation convenience, we can simplify (\ref{4.8}) as 
\begin{align}
(\textit{I}_{m+s}+\sigma\hat{A}Q\hat{A}^{T})\hat{d}_y=-\nabla\hat{\theta}(y),
\end{align}
where $\hat{A}=L^{-1}\bar{A}=[A;-\sqrt{1/\epsilon}B], \hat{d}_y=L^{-1}d_y$ and $\nabla\hat{\theta}(y)=L^{-1}\nabla\theta(y)$. Note that the cost of computing both $\bar{A}Q\bar{A}$ and $\hat{A}Q\hat{A}^{T}$ are $\mathcal{O}((m+s)^2n)$. Therefore, the matrix multiplication can be computational prohibitive when $n$ is large.
Fortunately, we can overcome this difficulty by exploiting the sparsity structure of our problem in the following manners.

\indent For the subdifferential of proximal mapping, from Li et al., \yrcite{li2018highly} we can always choose $Q\in \partial\text{Prox}_{\sigma\lambda\|x\|_{1}}(x)$ to be $Q=\text{diag}(q)$, a diagonal matrix whose $i$-th element is given by 
\begin{align}\label{4.10}
q_{i}=\begin{cases}
1,\qquad\text{if}\quad|x_{i}|>\sigma\lambda, \\
0,\qquad\text{otherwise}.
\end{cases}
\end{align}
Set $\mathcal{J}=\{j:|x_{j}|>\sigma\lambda\}$ with cardinality $|\mathcal{J}|=r$. By utilizing the diagonal structure of $Q$, we can write
\begin{align}
\bar{A}Q\bar{A}^{T} = \bar{A}_{\mathcal{J}}\bar{A}_{\mathcal{J}}^{T},\quad\hat{A}Q\hat{A}^{T}=\hat{A}_{\mathcal{J}}\hat{A}_{\mathcal{J}}^T,
\end{align}
where $\bar{A}_{\mathcal{J}}\in\mathbb{R}^{(m+s)\times r}$ is the submatrix of $\bar{A}$ with those columns contained in $\mathcal{J}$ preserved and the same for $\hat{A}_{\mathcal{J}}$.

Now we analyze the reduction of computational cost by exploring the second-order sparsity of the problem. By utilizing (\ref{4.10}), we can reduce the cost of computing $\bar{A}Q\bar{A}$ and $\hat{A}Q\hat{A}^{T}$ from $\mathcal{O}((m+s)^2n)$ to $\mathcal{O}((m+s)^2r)$. Due to the sparsity-inducing property of $p(x)=\lambda\|x\|_1$, $r$ is usually much smaller than $n$, we greatly reduce the computational cost. Consequently, the total computational cost of solving (\ref{4.7}) reduces from $\mathcal{O}((m+s)^2(m+s+n))$ to $\mathcal{O}((m+s)^2(m+s+r))$, meaning that the computational cost has no relationship with $n$. Thus, even for large dimension $n$, we can tackle the linear system (\ref{4.7}) by Cholesky factorization. In fact, when $r\ll m+s$, we can also directly invert the matrix using the Sherman-Morrison-Woodbury formula \cite{golub2012matrix} instead of using Cholesky factorization:
\begin{align*}
\ &(\textbf{I}_{m+s}+\sigma\hat{A}Q\hat{A}^{T})^{-1}=(\textbf{I}_{m+s}+\sigma\hat{A}_{\mathcal{J}}\hat{A}^{T}_{\mathcal{J}})^{-1}\\
=\ &\textbf{I}_{m+s}-\hat{A}_{\mathcal{J}}(\sigma^{-1}\textbf{I}_r+\hat{A}_{\mathcal{J}}^{T}\hat{A}_{\mathcal{J}})^{-1}\hat{A}_{\mathcal{J}}^{T}.
\end{align*}
As a result, the total computational cost to solve the Newton linear system can be reduced from $\mathcal{O}((m+s)^2(m+s+r))$ to
$\mathcal{O}(r^2(m+s+r))$. Although this drastic computational reduction seems exciting, in practice we have to determine when to solve the linear system by Cholesky factorization and when to compute the inverse directly in the aforementioned way. We balance these two choices by judging whether $r\leq \frac{1}{2}(m+s)$ holds.
Note that whichever way we choose to solve (\ref{4.7}), the computational cost only depends on $m+s$. Thus, when $m+s$ is not too large (smaller than $10^4$), we can always solve (\ref{4.7}) exactly by Cholesky factorization or by computing the inverse. Otherwise, we can choose CG to solve the Newton linear system inexactly.  

\section{Convergence analysis}\label{sec5}
In this section, we establish the convergence of both the augmented Lagrangian method and the semismooth method under some mild assumptions.
\subsection{Convergence of the augmented Lagrangian method}
We first introduce some definitions that are key to our analysis.

\indent Let $\mathcal{X}$ and $\mathcal{Y}$ be finite-dimensional Euclidean spaces. We say that $F:\mathcal{X}\rightrightarrows\mathcal{Y}$ is a multi-valued mapping if $F$ maps each element $x\in \mathcal{X}$ to a subset $F(x) \subseteq \mathcal{Y}$. Then, the graph of $F$ is defined as 
\begin{align*}
\text{gph}(F):=\{(x,y)\in\mathcal{X}\times\mathcal{Y}|y\in F(x)\},
\end{align*}
and the inverse mapping of $F$ is denoted by $F^{-1}$. The following definition is from Chapter 3 in \cite{dontchev2009implicit}.
\begin{define}
	A multi-valued mapping $F:\mathcal{X}\rightrightarrows\mathcal{Y}$ is said to be metrically subregular at $\bar{x}\in\mathcal{X}$ for $\bar{y}\in\mathcal{Y}$ with modulus $\kappa\geq0$, where $(\bar{x},\bar{y})\in\text{gph}(F)$, if there exist neighborhoods $\mathcal{E}_1$ of $\bar{x}$ and $\mathcal{E}_2$ of $\bar{y}$ such that
	\begin{align*}
	\text{dist}(x,F^{-1}(\bar{y}))\leq\kappa\text{dist}(\bar{y},F(x)\cap\mathcal{E}_2),\qquad\forall x\in\mathcal{E}_1.
	\end{align*}
\end{define}
The following result shows that for the constrained Lasso problem, the map $\mathcal{T}_\phi$ is metrically subregular at $\bar{x}$ for $0$. The proof can be found in the supplementary material.
\begin{thr} \label{thr1}
	Assume that $\mathcal{T}^{-1}_l(0)$ is non-empty and there exists $(\bar{u},\bar{v},\bar{w})\in\mathcal{T}_\psi^{-1}(0)$. For $h(\cdot)$ and $p(\cdot)$ chosen as in the constrained Lasso problem, the map $\mathcal{T}_\phi$ is metrically subregular at $\bar{x}$ for $0$..
\end{thr} 
Recall the stopping criteria (\ref{A}) and (\ref{B}).
Following Rockafellar \yrcite{rockafellar1976augmented}, we now establish the global convergence and local linear convergence of our inexact ALM algorithm under the metric subregularity of $\mathcal{T}_\phi$.
\begin{thr} \label{thr2}
	Assume that the solution set $\mathcal{T}^{-1}_{\phi}(0)$ to (\ref{P}) is nonempty. Let  $\{(z^{k},x^{k})\}$ be the sequence generated by Algorithm 1 with stopping criterion $(A)$. Then the sequence $\{x^{k}\}$ is bounded and converges to some $x^{\infty}\in\mathcal{T}^{-1}_{\phi}(0)$, and the sequence $\{z^{k}\}$ satisfies the following for all $k\geq0,z^{k}\in\mathcal{Z}$:
	\begin{align*}
	\ &\|A^{T}u^{k+1}-B^{T}v^{k+1}+w^{k+1}\|=\|x^{k+1}-x^{k}\|/\sigma_{k}\rightarrow 0, \\
	\ &g(z^{k+1})-\inf g \\ \leq\ &\Theta_k(z^{k+1})-\inf\Theta_k+(\|x^{k}\|^2-\|x^{k+1}\|^2)/2\sigma_{k},
	\end{align*}
	where $g(z)$ is the dual objective value defined in (\ref{D}) and $\Theta_{k}$ is defined in (\ref{3.1}). Moreover, if (\ref{D}) has a non-empty and bounded solution set, then the sequence $\{z^{k}\}$ is bounded and converges to an optimal solution to (\ref{D}). 
\end{thr}
\begin{thr} \label{thr3}
	Suppose that $\mathcal{T}_\phi$ is metrically subregular at $x^\infty$ for $0$ with modulus $\kappa_\phi$. Suppose further that the sequence $\{z^{k},x^{k}\}$ is generated under the criterion $(B)$. Then, for all sufficiently large $k$, we have
	\begin{align*}
	\ &\text{dist}(x^{k+1},\mathcal{T}_{\phi}^{-1}(0))\leq\rho_{k}\text{dist}(x^{k},\mathcal{T}_{\phi}^{-1}(0)), \\
	\ &\|A^{T}u^{k+1}-B^{T}v^{k+1}+w^{k+1}\|\leq\varsigma_{k}\text{dist}(x^{k},\mathcal{T}_{\phi}^{-1}(0)), \\
	\ &g(z^{k+1})-\inf g\leq\varsigma'_{k}\text{dist}(x^{k},\mathcal{T}_{\phi}^{-1}(0)),
	\end{align*}
	where 
	\begin{align*}
	\ &1>\rho_{k}=(\kappa_{\phi}/\sqrt{\kappa_{\phi}^2+\sigma_{k}^2}+2\zeta_k)(1-\zeta_k^{-1})\\ \  &\rightarrow\rho_{\infty}=\kappa_{\phi}/\sqrt{\kappa_{\phi}^2+\sigma_{\infty}^2}\quad(\rho_{\infty}=0\text{ if }\sigma_{\infty}=\infty), \\
	\ &\varsigma_k = \sigma_{k}^{-1}(1-\zeta_k)^{-1}\rightarrow\varsigma_{\infty}=1/\sigma_{\infty}\quad(\varsigma_{\infty}=0\text{ if }\sigma_{\infty}=\infty), \\
	\ &\varsigma_k '=\varsigma_k(\zeta_{k}^2\|x^{k+1}-x^{k}\|+\|x^{k+1}\|+\|x^{k}\|)/2\\ \ &\rightarrow\varsigma_{\infty}'=\|x^{\infty}\|/\sigma_{\infty}\quad(\varsigma_{\infty}'=0\text{ if }\sigma_{\infty}=\infty),
	\end{align*}  
	and $\zeta_k$ comes from stopping criterion (\ref{B}) and $g(z)$ is the dual objective value defined  in (\ref{D}).
\end{thr}
\begin{rmk}
	Under the assumptions in Theorem \ref{thr2}, we establish the global convergence of the sequence $\{(x^{k},z^{k})\}$ corresponding to the primal and dual problems. Theorem \ref{thr1} shows that $\mathcal{T}_{\phi}$ is metrically subregular for the constrained Lasso problem. Thus, it satisfies the assumptions in Theorem \ref{thr3}. Moreover, Theorem \ref{thr3} shows that the KKT residuals corresponding to (\ref{P}) and (\ref{D}) converge superlinearly.
\end{rmk}

\subsection{Convergence of the semismooth Newton method}
We first establish the convergence of Algorithm \ref{alg2} under the following mild assumption:
\begin{asup} \label{asup1}
	The matrix $B$ has full row rank: $\text{rank}(B)=s$.
\end{asup} 
Using Assumption \ref{asup1} and following the proof of Theorem 3.4 of \cite{zhao2010newton}, we obtain the following result.
\begin{thr} \label{thr4}
	Suppose the Assumption \ref{asup1} holds. Then Algorithm \ref{alg2} is well defined and any accumulation point $(\hat{u},\hat{v})$ is an optimal solution to problem (\ref{4.3}).
\end{thr}
Next, we introduce the notion of \emph{constraint nondegeneracy}. 
\begin{define}
	Let $z\in\mathbb{R}^{n}$ and $\hat{Q}=\text{diag}(\hat{q})\in\partial\text{Prox}_{\sigma\lambda\|x\|_1}(z)$ with $\hat{q}$ defined in (\ref{4.10}). We say that the \textbf{constraint nondegeneracy condition} holds at $z$ if
	\begin{align}
	\text{lin}(B\hat{Q})=\mathbb{R}^s
	\end{align} 
	holds at $z$ where $\text{lin}(B\hat{Q})$ denotes the linear space of $B\hat{Q}$.
\end{define}
Now we can establish the superlinear convergence of Algorithm \ref{alg2}. The proof of the following result can be found in the supplementary material.
\begin{thr} \label{thr5}
	Let $\{\hat{u},\hat{v}\}$ be an accumulation point of the sequence $\{(u^{j},v^{j})\}$ generated by Algorithm \ref{alg2}. Suppose that the constraint nondegeneracy condition holds at $\hat{z}=x-\sigma(A^T\hat{u}-B^{T}\hat{v})$. Then, the sequence  $\{(u^j,v^j)\}$ converges to $\{(\hat{u},\hat{v})\}$ and 
	\begin{align*}
	\|(u^{j+1},v^{j+1})-(\hat{u},\hat{v})\|=\mathcal{O}(	\|(u^{j},v^{j})-(\hat{u},\hat{v})\|^{1+\tau}).
	\end{align*}
\end{thr}

\section{Numerical experiments}\label{sec6}
In this section, we evaluate the performance of our algorithm to solve the constrained Lasso problem on both synthetic and real datasets. 
We compare with four state-of-the-art methods: primal-dual method \cite{chambolle2011first}, linearized augmented Lagrangian method \cite{yang2013linearized}, ADMM \cite{boyd2011distributed}, accelerated ADMM (A-ADMM) \cite{goldstein2014fast}. 
For both the ADMM and A-ADMM method, we set the step size to be 1.618. 

\indent We set the penalty parameter $\lambda$ in the constrained Lasso problem as
$\lambda = \lambda_l\|A^{T}b\|_{\infty} $,
where $0<\lambda_l<1$. In the numerical experiments, we measure the accuracy of solution $\{x,u,v,w\}$ generated by our algorithm by the following residual:
\begin{align}
\eta_{\text{cLasso}} =\max\{\eta_{P},\eta_{D}\},
\end{align}
where 
\begin{align*}
\eta_{P}=\dfrac{\|Bx-d\|}{1+\|d\|},\quad\eta_{D}=\|A^Tu-B^Tv+w\|
\end{align*}
are the primal and dual feasibility. Moreover, we compute the relative primal dual gap defined by
\begin{align*}
\eta_{\text{relgap}}=\dfrac{\text{obj}_P-\text{obj}_D}{1+|\text{obj}_P|+|\text{obj}_D|},
\end{align*}
where $\text{obj}_P=\frac{1}{2}\|Ax-b\|^2+\lambda\|x\|_1$ and $\text{obj}_D=-(\frac{1}{2}\|u\|^2+\langle b,u\rangle-\langle d,v\rangle)$. We stop our algorithm  when $\eta_{\text{cLasso}}<\varepsilon$ for a given tolerance $\varepsilon$ and stop other compared algorithms when both the primal and dual residuals are smaller than $\varepsilon$. Before the comparison, we run our algorithm with high accuracy $\varepsilon=10^{-10}$ and set this optimal value as the baseline. The optimality gap is measured by $\eta_{\text{gap}}=f(x)-f(x^{*})$. For our numerical experiments, we set $\varepsilon=10^{-6}$. All the algorithms will be stopped when they reach the maximum iteration number, which is set at 100 for our algorithm and at 10000 for other algorithms. The codes were written in MATLAB and run on a server with 10 cores with 20 threads Intel Xeon E7-4870 CPU at 2.4 GHz with 1000 GB memory. 

\indent Our numerical tests have the following scenarios on both synthetic and real datasets: 
\begin{itemize}	
	\setlength{\itemsep}{0.6pt}
	\setlength{\parskip}{0.6pt}
	\item sum to zero constraints,
	\item $B$ and $d$ are randomly generated,
	\item transform generalized Lasso problem to an equivalent constrained Lasso problem.
\end{itemize}
Due to limited space, we only present the first and last ones in this paper, and full details of the experiments for all three scenarios are deferred to the supplementary material.
\subsection{Synthetic data}
In this section we display the performance of our algorithm on synthetic datasets. For synthetic data, we generate $A\in\mathbb{R}^{m\times n}$ from independent and identical (iid) standard normal distribution and $b=A\mathring{x}+\varrho$, where $\varrho\in N(0,0.001*\textbf{I}_m)$ and $\mathring{x}$ is a sparse vector. We set $n=10m$ with $m=200,300,500,800,1000$ and $\lambda_l=10^{-2},10^{-3},10^{-4}$.
\subparagraph{Sum to zero constraints}
In this subsection we consider the sum to zero constraints, i.e., $e^{T}x=0$. Recently this problem generates much interest in the statistics and bioinformatics communities \cite{altenbuchinger2016reference,lin2014variable,shi2016regression}. 

\indent Table \ref{tab1} reports the performance of all methods we test on synthetic datasets. Note that $m$ is the sample size and $n$ is the dimension of each sample. Moreover, "nnz" denotes the number of nonzeros in the solution $x$ obtained by our method which is defined as follows:
\begin{align*}
\text{nnz}:=\min\{k|\sum_{i=1}^{k}|x^{\downarrow}_k|\geq 0.999\|x\|_1\},
\end{align*}
where $x^{\downarrow}$ is sorted in the decreasing order of absolute value of $x$ such that $|x^{\downarrow}_1|\geq|x^{\downarrow}_2|\geq\dots\geq|x^{\downarrow}_n|$. 

\begin{table}[t] 
	\small
	\caption{\small Performance of our SSNAL method(a), primal dual method(b), linearized ALM(c), ADMM(d) and A-ADMM(e) with sum to zero constraints on synthetic data sets. 'nnz' denotes the number of nonzeros of the solution obtained by our algorithm.}	
	\vskip 0.1in
	\label{tab1}
	\begin{tabular}{cccccccc}
		\hline  & $\lambda_l$& nnz & \multicolumn{5}{c}{ running time (seconds)}   \\
		\hline size $m;n$ & & & a&b&c&d&e \\
		\hline\multirow{3}{*}{200;2000} & $10^{-2}$& 189& \textbf{0.4}&5.3&5.0&6.2&8.3\\
		& $10^{-3}$&193&\textbf{1.0}&6.5&4.6&5.9&7.9\\
		&$10^{-4}$&195&\textbf{1.3}&6.3&6.7&7.3&6.4   \\
		\hline\multirow{3}{*}{300;3000} & $10^{-2}$& 286& \textbf{1.2}&19&18&22&25\\
		& $10^{-3}$&289&\textbf{2.3}&18&17&23&22\\
		&$10^{-4}$&291&\textbf{3.2}&18&25&24&24   \\
		\hline\multirow{3}{*}{500;5000} & $10^{-2}$&471&\textbf{2.7}&76&80&155&162 \\
		& $10^{-3}$&481&\textbf{3.8}&70&95&158&158 \\
		&$10^{-4}$&484&\textbf{5.3}&78&81&155&159\\
		\hline
		\multirow{3}{*}{800;8000} & $10^{-2}$&746&\textbf{6.2}&151&149&355&359\\
		& $10^{-3}$&760&\textbf{10}&153&159&357&361\\
		&$10^{-4}$& 764&\textbf{15}&157&169&355&363 \\
		\hline
		\multirow{3}{*}{1000;10000} & $10^{-2}$&935&\textbf{9.0}&225&249&551&553\\
		& $10^{-3}$&958&\textbf{18}&252&263&546&556\\
		&$10^{-4}$&964&\textbf{23}&236&231&550&556  \\
		\hline		
	\end{tabular}
	\vskip -0.1in
\end{table}

\begin{figure}[!hbt] 
	\vskip 0.1in
	\begin{minipage}[t]{0.48\linewidth}
		\includegraphics[width=1.0\textwidth]{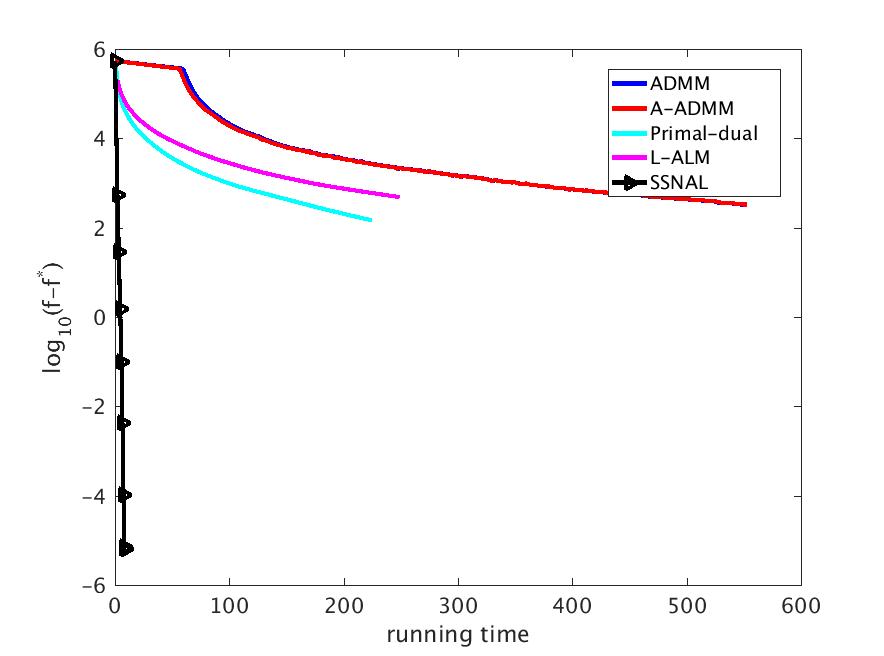}
	\end{minipage}
	\begin{minipage}[t]{0.48\linewidth}
		\includegraphics[width=1.0\textwidth]{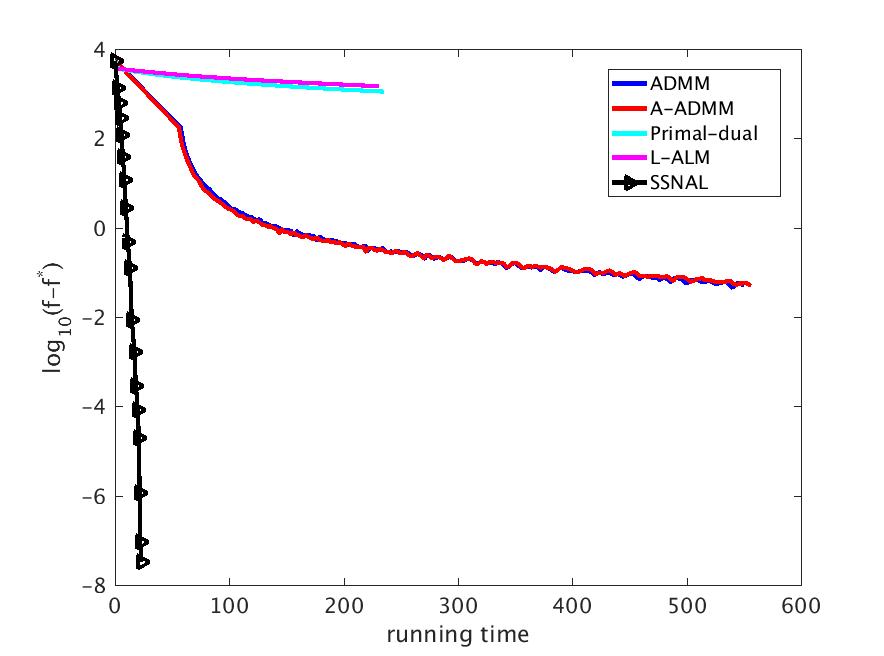}
	\end{minipage}
	\caption{Constrained Lasso with sum to zero constraints, $m=1000, n=10000$ and $\lambda_l=10^{-2},10^{-4}$.}
	\label{fig1}
	\vskip -0.1in	
\end{figure}
\indent Figure \ref{fig1} plots the optimality gap with running time for the case of $m=1000, n=10000$ and $\lambda_l=10^{-2},10^{-4}$. From the figures, it is easy to observe that our algorithm is much faster than others and achieve a more accurate solution. 

\subparagraph{Generalized Lasso}
In this section, we first transform the generalized Lasso problem to an equivalent constrained Lasso problem using techniques from \cite{gaines2018algorithms}. 

Recall the generalized Lasso model in (\ref{1.3}). If $\text{rank}(D)=n$ and set $p=n+s$, a singular value decomposition of $D$ is given as follows:
\begin{align*}
D=U\Sigma V^T=[U_1,U_2]\begin{bmatrix}
\Sigma_1  \\
\textbf{0} 
\end{bmatrix}V_1^T ,
\end{align*}
where $U_1\in\mathbb{R}^{p\times n}$, $U_2\in\mathbb{R}^{p\times s}$, $\Sigma_1\in\mathbb{R}^{n\times n}$, $V_1\in\mathbb{R}^{n\times n}$ and $\textbf{0}\in\mathbb{R}^{s\times n}$. Then (\ref{1.3}) is equivalent to the constrained Lasso problem:
\begin{align} \label{5.3}
\begin{split}
\min_{z}\ & \frac{1}{2}\|AD^{\dagger}z-b\|^2+\lambda\|z\|_1 \\
\text{s.t.}\ &U_2^Tz=0,
\end{split}
\end{align}
where $D^{\dagger}=V_1\Sigma_1^{-1}U_1^T\in\mathbb{R}^{n\times p}$ , $z\in\mathbb{R}^{p}$ and we can recover the solution $x^{*}$ to (\ref{1.3}) via $\bar{x}=D^{\dagger}\bar{z}$, where $\bar{z}$ is the optimal solution of (\ref{5.3}). In our experiments, we construct $D=\begin{bmatrix}
D_1 \\ D_2
\end{bmatrix}$, where $D_1$ is an $n\times n$ identity matrix and $D_2$ is an $s\times n$ random matrix.

\indent In Table \ref{tab3} we summarize the numerical results on Problem (\ref{5.3}). Observe that our algorithm is 5-10 times faster than other algorithms for all choices of $\lambda_l$. 
\begin{table}[ht] 
	\small
	\caption{\small Performance of our SSNAL method(a), primal dual method(b), linearized ALM(c), ADMM(d) and A-ADMM(e) with generalized Lasso problem on synthetic datasets.}
	\vskip 0.1in
	\label{tab3}
	\begin{tabular}{cccccccc}
		\hline  & $\lambda_l$& nnz &\multicolumn{5}{c}{ running time (seconds) }   \\
		\hline size $m;n$ & & &a&b&c&d&e \\
		\hline\multirow{3}{*}{200;2000} & $10^{-2}$&221&\textbf{3.7}&8.7&18&25&27\\
		& $10^{-3}$&220&\textbf{1.9}&8.5&17&23&25\\
		&$10^{-4}$&219&\textbf{2.2}&10&17&24&24  \\
		\hline\multirow{3}{*}{300;3000} & $10^{-2}$&311&\textbf{4.6}&14&28&43&43\\
		& $10^{-3}$&317&\textbf{3.0}&13&30&43&44\\
		&$10^{-4}$&318&\textbf{2.6}&13&32&41&44  \\
		\hline\multirow{3}{*}{500;5000} & $10^{-2}$&495&\textbf{12}&74&111&176&179\\
		& $10^{-3}$&505&\textbf{10}&77&113&180&179 \\
		&$10^{-4}$&505&\textbf{7.5}&85&111&177&179\\
		\hline
		\multirow{3}{*}{800;8000} & $10^{-2}$&779&\textbf{27}&152&200&383&387\\
		& $10^{-3}$&789&\textbf{24}&158&216&383&386\\
		&$10^{-4}$&795&\textbf{17}&152&210&390&392  \\
		\hline
		\multirow{3}{*}{1000;10000} & $10^{-2}$&975&\textbf{51}&254&311&589&589\\
		& $10^{-3}$&993&\textbf{41}&234&325&597&595\\
		&$10^{-4}$&995&\textbf{27}&250&322&589&587  \\
		\hline		
	\end{tabular}
	\vskip -0.1in
\end{table}
 
\begin{figure}[ht]
	\vskip 0.1in 
	\begin{minipage}[t]{0.48\linewidth}
		\includegraphics[width=1.0\textwidth]{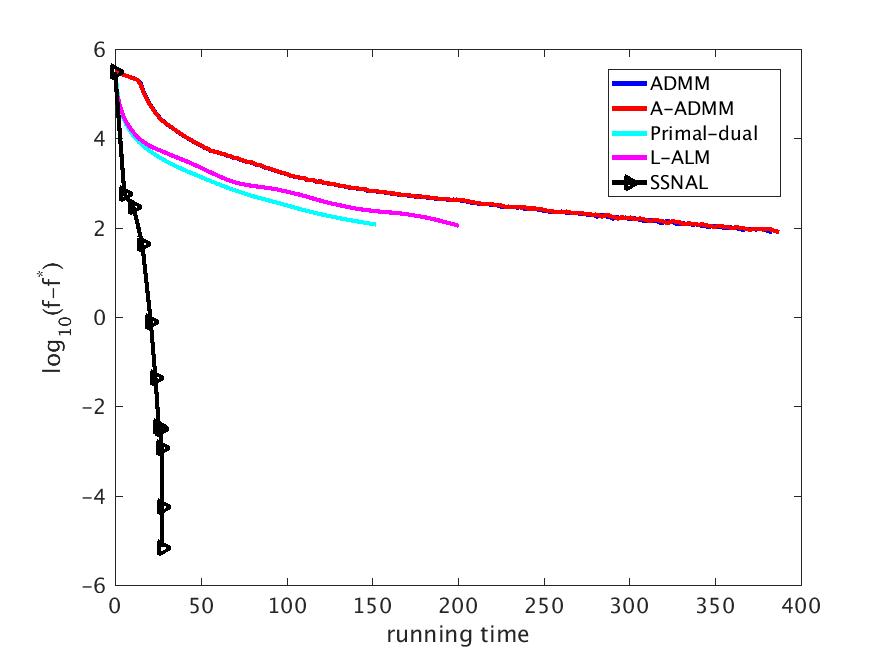}
	\end{minipage}
	\begin{minipage}[t]{0.48\linewidth}
		\includegraphics[width=1.0\textwidth]{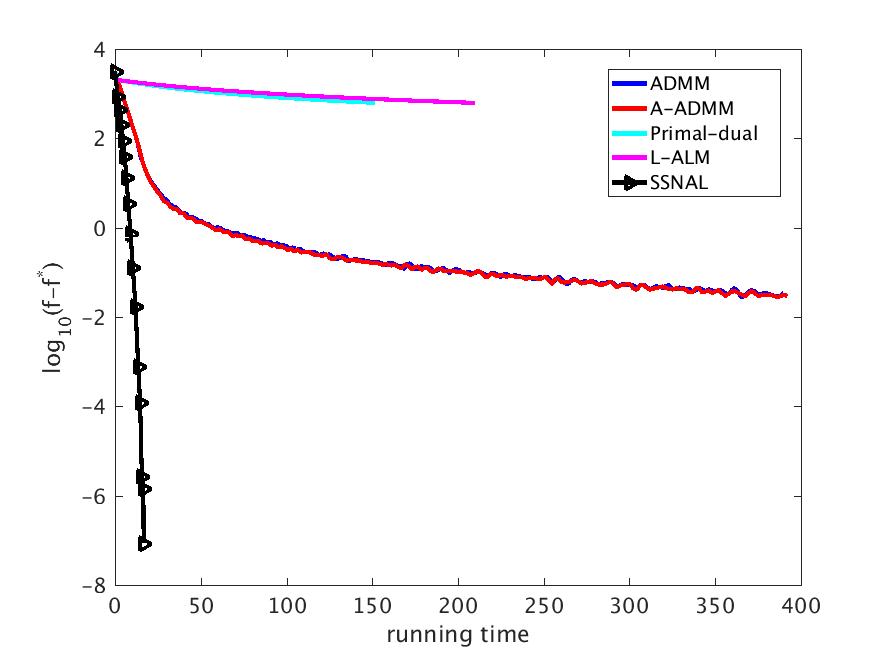}
	\end{minipage}
	\caption{Constrained Lasso with generalized Lasso problem, $m=800, n=8000, s=30$ and $\lambda_l=10^{-2},10^{-4}$.}
	\label{fig3}
	\vskip -0.1in	
\end{figure}

\indent We plot the optimality gap with running time in Figure \ref{fig3} for the case $m=800, n=8000$ and $s=30$. The figures demonstrate that our algorithm is superior to other methods both in terms of solution accuracy and running time. 

\subsection{Real Data}
In this section we test all algorithms on LIBSVM regression datasets \cite{chang2011libsvm}. We preprocess the dataset \textbf{abalone, bodyfat, housing, mpg} and \textbf{space\_ga} to expand the original features based on polynomial basis functions as stated in \cite{huang2010predicting}. For example, \textbf{abalone7} means that we expand the feature of \textbf{abalone} by an order 7 polynomial basis function. As mentioned before, we run numerical experiments with the sum to zero constraints and generalized Lasso problem and set $\lambda_l=10^{-3}$ and $10^{-4}$.

\subparagraph{Sum to zero constraints}
We consider the sum to zero constraints and summarize our experimental results on UCI regression datasets in Table \ref{tab4}. 
\indent Moreover, we display the details of running time with optimality gap on \textbf{bodyfat5} dataset in Figure \ref{fig4}, which highlight that our algorithm attains a more accurate solution much faster than other algorithms. 
\begin{table}[!hbp] 
	\small
	\caption{\small Performance of our SSNAL method(a), primal dual method(b), linearized ALM(c), ADMM(d) and A-ADMM(e) with sum to zero constraints on UCI regression datasets.}
	\vskip 0.1in
	\label{tab4}
	\begin{tabular}{cccccccc}
		\hline  & $\lambda_l$& nnz &\multicolumn{5}{c}{ running time (seconds)}   \\
		\hline problem name  & & & a&b&c&d&e \\
		$m;n$ &&&\multicolumn{5}{c}{}\\
		\hline
		abalone7&$10^{-3}$&23&\textbf{21}&105&115&250&247 \\
		4177;6435&$10^{-4}$&63&\textbf{30}&104&118&165&169\\
		\hline
		bodyfat5&$10^{-3}$&39&\textbf{3.6}&302&296&534&544 \\
		252;11628&$10^{-4}$&79&\textbf{4.3}&296&298&235&229 \\
		\hline
		housing5&$10^{-3}$&113&\textbf{5.0}&188&185&371&365 \\
		506;8568&$10^{-4}$&216&\textbf{6.7}&179&189&121&96\\
		\hline 
		mpg7 & $10^{-3}$&43&\textbf{1.3}&35&36&19&18\\
		392;3432	& $10^{-4}$&132&\textbf{2.2}&35&40&9.1&7.5\\
		\hline
		space\_ga9 & $10^{-3}$&14&\textbf{6.5}&78&89&161&167 \\
		3107;5005 &$10^{-4}$&36&\textbf{11}&80&97&50&40 \\
		\hline		
	\end{tabular}
	\vskip -0.1in
\end{table}
\begin{figure}[t] 
	\vskip 0.1in
	\begin{minipage}[!hbp]{0.48\linewidth}
		\includegraphics[width=1.0\textwidth]{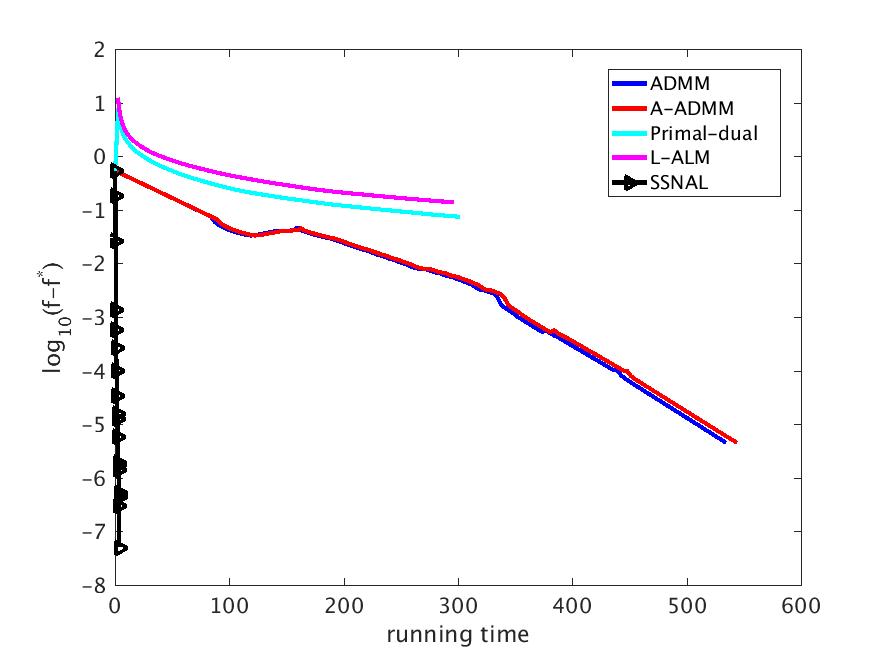}
	\end{minipage}
	\begin{minipage}[!hbp]{0.48\linewidth}
		\includegraphics[width=1.0\textwidth]{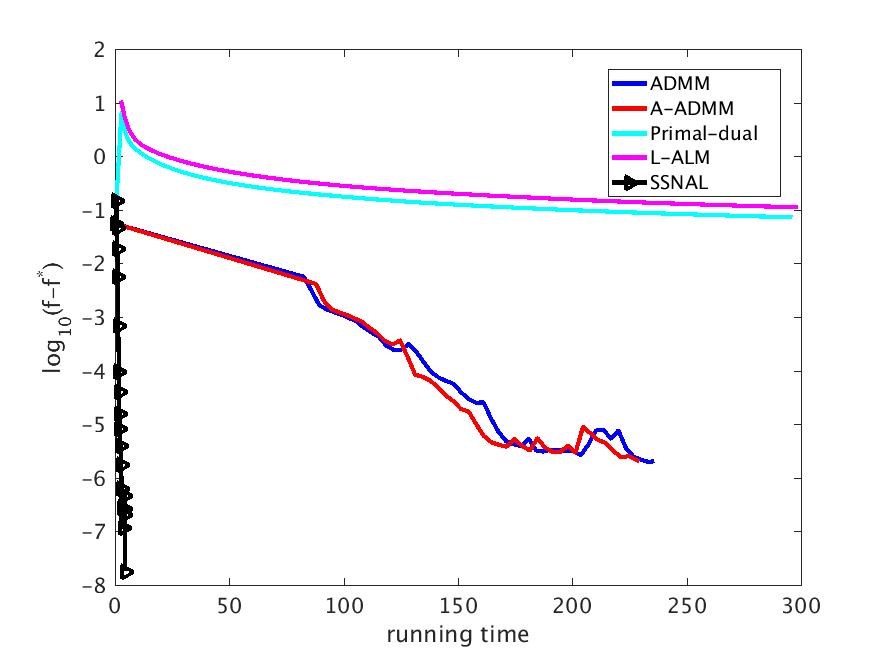}
	\end{minipage}
	\caption{Constrained Lasso with sum to zero constraints, $\lambda_l=10^{-3},10^{-4}$ on  \textbf{bodyfat5} dataset.}
	\label{fig4}
	\vskip -0.1in	
\end{figure}

\subparagraph{Generalized Lasso}
For the transformed generalized Lasso problem (\ref{5.3}), we test on real datasets and present numerical results  in Table \ref{tab6} and 
 plot the optimality gap with running time on \textbf{mpg7} dataset in Figure \ref{fig6} . One can observe that our algorithm outperforms other first-order methods both in solution accuracy and running time.
\begin{figure}[t] 
	\vskip 0.1in
	\begin{minipage}[!hbp]{0.48\linewidth}
		\includegraphics[width=1.0\textwidth]{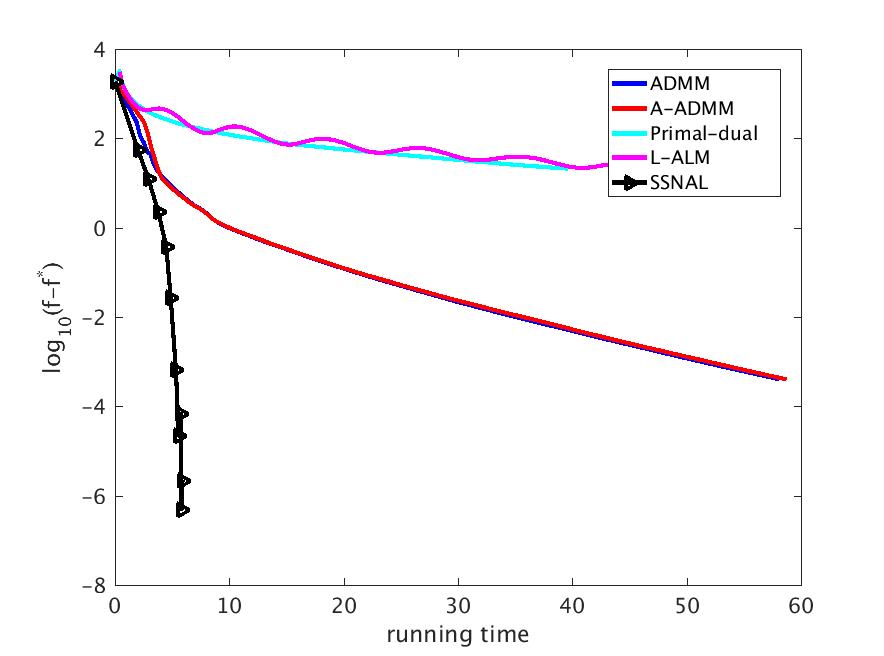}
	\end{minipage}
	\begin{minipage}[!hbp]{0.48\linewidth}
		\includegraphics[width=1.0\textwidth]{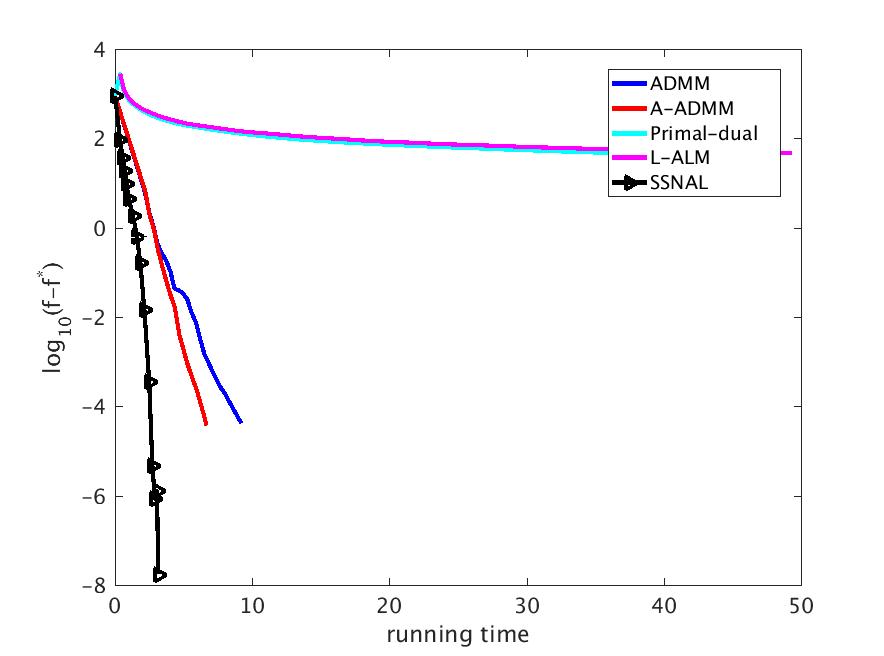}
	\end{minipage}
	\caption{Constrained Lasso with generalized Lasso problem, $\lambda_l=10^{-3},10^{-4}$ on  \textbf{mpg7} dataset.}
	\label{fig6}
	\vskip -0.1in	
\end{figure}
\begin{table}[!htb] 
	\small
	\caption{\small Performance of our SSNAL method(a), primal dual method(b), linearized ALM(c), ADMM(d) and A-ADMM(e) with generalized Lasso problem on UCI regression datasets.}
	\vskip 0.1in
	\label{tab6}
	\begin{tabular}{cccccccc}
		\hline  & $\lambda_l$& nnz &\multicolumn{5}{c}{running time (seconds)}   \\
		\hline problem name  & & & a&b&c&d&e \\
		$m;n$ &&&\multicolumn{5}{c}{}\\
		\hline
		abalone7&$10^{-3}$&46&\textbf{51}&127&165&267&270 \\
		4177;6435&$10^{-4}$&78&\textbf{56}&114&176&98&84\\
		\hline
		bodyfat5&$10^{-3}$&44&\textbf{3.7}&321&395&772&783\\
		252;11628&$10^{-4}$&68&\textbf{4.3}&309&393&121&93 \\
		\hline
		housing5&$10^{-3}$&105&\textbf{11}&198&246&458&463 \\
		506;8568&$10^{-4}$&219&\textbf{7.7}&194&257&153&146\\
		\hline 
		mpg7 & $10^{-3}$&61&\textbf{5.9}&40&47&58&59\\
		392;3432	& $10^{-4}$&127&\textbf{3.1}&37&49&9.2&6.7\\
		\hline
		space\_ga9 & $10^{-3}$&47&\textbf{15}&83&116&182&180 \\
		3107;5005 &$10^{-4}$&78&\textbf{15}&83&111&30&24\\
		\hline		
	\end{tabular}
	\vskip -0.1in
\end{table}

\section{Conclusion}
In this paper we propose a semismooth Newton augmented Lagrangian method to solve the dual problem of the constrained Lasso problem. Under some mild assumptions, we establish the superlinear convergence for both outer loop and inner subproblem solver. By exploiting the sparse structure of the problem, we provide  efficient implementations to solve the subproblem and  greatly reduce the computational cost. We present extensive numerical experiments to show the efficiency and robustness of our algorithm on both synthetic and real datasets. As a future work, we plan to extend our algorithmic framework to handle more general constraints. Moreover, stochastic and distributed versions of this algorithmic framework are also interesting.

\newpage
\bibliography{references2019}

\begin{thebibliography}{42}
\providecommand{\natexlab}[1]{#1}
\providecommand{\url}[1]{\texttt{#1}}
\expandafter\ifx\csname urlstyle\endcsname\relax
  \providecommand{\doi}[1]{doi: #1}\else
  \providecommand{\doi}{doi: \begingroup \urlstyle{rm}\Url}\fi

\bibitem[Altenbuchinger et~al.(2016)Altenbuchinger, Rehberg, Zacharias,
  St{\"a}mmler, Dettmer, Weber, Hiergeist, Gessner, Holler, Oefner,
  et~al.]{altenbuchinger2016reference}
Altenbuchinger, M., Rehberg, T., Zacharias, H., St{\"a}mmler, F., Dettmer, K.,
  Weber, D., Hiergeist, A., Gessner, A., Holler, E., Oefner, P.~J., et~al.
\newblock Reference point insensitive molecular data analysis.
\newblock \emph{Bioinformatics}, 33\penalty0 (2):\penalty0 219--226, 2016.

\bibitem[Arag{\'o}n~Artacho \& Geoffroy(2008)Arag{\'o}n~Artacho and
  Geoffroy]{aragon2008characterization}
Arag{\'o}n~Artacho, F.~J. and Geoffroy, M.~H.
\newblock Characterization of metric regularity of subdifferentials.
\newblock \emph{Journal of Convex Analysis}, 15\penalty0 (2):\penalty0
  365--380, 2008.

\bibitem[Bauschke \& Borwein(1996)Bauschke and Borwein]{bauschke1996projection}
Bauschke, H.~H. and Borwein, J.~M.
\newblock On projection algorithms for solving convex feasibility problems.
\newblock \emph{SIAM review}, 38\penalty0 (3):\penalty0 367--426, 1996.

\bibitem[Bauschke et~al.(1999)Bauschke, Borwein, and Li]{bauschke1999strong}
Bauschke, H.~H., Borwein, J.~M., and Li, W.
\newblock Strong conical hull intersection property, bounded linear regularity,
  jameson’s property (g), and error bounds in convex optimization.
\newblock \emph{Mathematical Programming}, 86\penalty0 (1):\penalty0 135--160,
  1999.

\bibitem[Beck \& Teboulle(2009)Beck and Teboulle]{beck2009fast}
Beck, A. and Teboulle, M.
\newblock A fast iterative shrinkage-thresholding algorithm for linear inverse
  problems.
\newblock \emph{SIAM journal on imaging sciences}, 2\penalty0 (1):\penalty0
  183--202, 2009.

\bibitem[Bertsekas(2014)]{bertsekas2014constrained}
Bertsekas, D.~P.
\newblock \emph{Constrained optimization and Lagrange multiplier methods}.
\newblock Academic press, 2014.

\bibitem[Boyd et~al.(2011)Boyd, Parikh, Chu, Peleato, Eckstein,
  et~al.]{boyd2011distributed}
Boyd, S., Parikh, N., Chu, E., Peleato, B., Eckstein, J., et~al.
\newblock Distributed optimization and statistical learning via the alternating
  direction method of multipliers.
\newblock \emph{Foundations and Trends{\textregistered} in Machine learning},
  3\penalty0 (1):\penalty0 1--122, 2011.

\bibitem[Burnham \& Anderson(2003)Burnham and Anderson]{burnham2003model}
Burnham, K.~P. and Anderson, D.~R.
\newblock \emph{Model selection and multimodel inference: a practical
  information-theoretic approach}.
\newblock Springer Science \& Business Media, 2003.

\bibitem[Cand{\`e}s \& Wakin(2008)Cand{\`e}s and Wakin]{candes2008introduction}
Cand{\`e}s, E.~J. and Wakin, M.~B.
\newblock An introduction to compressive sampling.
\newblock \emph{IEEE signal processing magazine}, 25\penalty0 (2):\penalty0
  21--30, 2008.

\bibitem[Chambolle \& Pock(2011)Chambolle and Pock]{chambolle2011first}
Chambolle, A. and Pock, T.
\newblock A first-order primal-dual algorithm for convex problems with
  applications to imaging.
\newblock \emph{Journal of mathematical imaging and vision}, 40\penalty0
  (1):\penalty0 120--145, 2011.

\bibitem[Chang \& Lin(2011)Chang and Lin]{chang2011libsvm}
Chang, C.-C. and Lin, C.-J.
\newblock Libsvm: a library for support vector machines.
\newblock \emph{ACM transactions on intelligent systems and technology (TIST)},
  2\penalty0 (3):\penalty0 27, 2011.

\bibitem[Chen et~al.(2001)Chen, Donoho, and Saunders]{chen2001atomic}
Chen, S.~S., Donoho, D.~L., and Saunders, M.~A.
\newblock Atomic decomposition by basis pursuit.
\newblock \emph{SIAM review}, 43\penalty0 (1):\penalty0 129--159, 2001.

\bibitem[Clarke(1990)]{clarke1990optimization}
Clarke, F.~H.
\newblock \emph{Optimization and nonsmooth analysis}, volume~5.
\newblock Siam, 1990.

\bibitem[Dontchev \& Rockafellar(2009)Dontchev and
  Rockafellar]{dontchev2009implicit}
Dontchev, A.~L. and Rockafellar, R.~T.
\newblock Implicit functions and solution mappings.
\newblock \emph{Springer Monogr. Math.}, 2009.

\bibitem[Fan \& Li(2001)Fan and Li]{fan2001variable}
Fan, J. and Li, R.
\newblock Variable selection via nonconcave penalized likelihood and its oracle
  properties.
\newblock \emph{Journal of the American statistical Association}, 96\penalty0
  (456):\penalty0 1348--1360, 2001.

\bibitem[Gaines et~al.(2018)Gaines, Kim, and Zhou]{gaines2018algorithms}
Gaines, B.~R., Kim, J., and Zhou, H.
\newblock Algorithms for fitting the constrained lasso.
\newblock \emph{Journal of Computational and Graphical Statistics}, \penalty0
  (just-accepted), 2018.

\bibitem[Goldstein et~al.(2014)Goldstein, O'Donoghue, Setzer, and
  Baraniuk]{goldstein2014fast}
Goldstein, T., O'Donoghue, B., Setzer, S., and Baraniuk, R.
\newblock Fast alternating direction optimization methods.
\newblock \emph{SIAM Journal on Imaging Sciences}, 7\penalty0 (3):\penalty0
  1588--1623, 2014.

\bibitem[Golub \& Van~Loan(2012)Golub and Van~Loan]{golub2012matrix}
Golub, G.~H. and Van~Loan, C.~F.
\newblock \emph{Matrix computations}, volume~3.
\newblock JHU Press, 2012.

\bibitem[Hiriart-Urruty et~al.(1984)Hiriart-Urruty, Strodiot, and
  Nguyen]{hiriart1984generalized}
Hiriart-Urruty, J.-B., Strodiot, J.-J., and Nguyen, V.~H.
\newblock Generalized hessian matrix and second-order optimality conditions for
  problems withc 1, 1 data.
\newblock \emph{Applied mathematics and optimization}, 11\penalty0
  (1):\penalty0 43--56, 1984.

\bibitem[Huang et~al.(2010)Huang, Jia, Yu, Chun, Maniatis, and
  Naik]{huang2010predicting}
Huang, L., Jia, J., Yu, B., Chun, B.-G., Maniatis, P., and Naik, M.
\newblock Predicting execution time of computer programs using sparse
  polynomial regression.
\newblock In \emph{Advances in neural information processing systems}, pp.\
  883--891, 2010.

\bibitem[James et~al.(2013)James, Paulson, and
  Rusmevichientong]{james2013penalized}
James, G.~M., Paulson, C., and Rusmevichientong, P.
\newblock Penalized and constrained regression.
\newblock \emph{Unpublished Manuscript, available at http://www-bcf. usc.
  edu/\~{} gareth/research/Research. html}, 2013.

\bibitem[Li et~al.(2015)Li, Sun, and Toh]{li2015qsdpnal}
Li, X., Sun, D., and Toh, K.-C.
\newblock Qsdpnal: A two-phase proximal augmented lagrangian method for convex
  quadratic semidefinite programming.
\newblock \emph{arXiv preprint arXiv:1512.08872}, pp.\  1--35, 2015.

\bibitem[Li et~al.(2018{\natexlab{a}})Li, Sun, and Toh]{li2018efficiently}
Li, X., Sun, D., and Toh, K.-C.
\newblock On efficiently solving the subproblems of a level-set method for
  fused lasso problems.
\newblock \emph{SIAM Journal on Optimization}, 28\penalty0 (2):\penalty0
  1842--1866, 2018{\natexlab{a}}.

\bibitem[Li et~al.(2018{\natexlab{b}})Li, Sun, and Toh]{li2018highly}
Li, X., Sun, D., and Toh, K.-C.
\newblock A highly efficient semismooth newton augmented lagrangian method for
  solving lasso problems.
\newblock \emph{SIAM Journal on Optimization}, 28\penalty0 (1):\penalty0
  433--458, 2018{\natexlab{b}}.

\bibitem[Lin et~al.(2014)Lin, Shi, Feng, and Li]{lin2014variable}
Lin, W., Shi, P., Feng, R., and Li, H.
\newblock Variable selection in regression with compositional covariates.
\newblock \emph{Biometrika}, 101\penalty0 (4):\penalty0 785--797, 2014.

\bibitem[Luo \& Tseng(1992)Luo and Tseng]{luo1992linear}
Luo, Z.-Q. and Tseng, P.
\newblock On the linear convergence of descent methods for convex essentially
  smooth minimization.
\newblock \emph{SIAM Journal on Control and Optimization}, 30\penalty0
  (2):\penalty0 408--425, 1992.

\bibitem[Meinshausen(2007)]{meinshausen2007relaxed}
Meinshausen, N.
\newblock Relaxed lasso.
\newblock \emph{Computational Statistics \& Data Analysis}, 52\penalty0
  (1):\penalty0 374--393, 2007.

\bibitem[Mifflin(1977)]{mifflin1977semismooth}
Mifflin, R.
\newblock Semismooth and semiconvex functions in constrained optimization.
\newblock \emph{SIAM Journal on Control and Optimization}, 15\penalty0
  (6):\penalty0 959--972, 1977.

\bibitem[Nesterov(2013)]{nesterov2013gradient}
Nesterov, Y.
\newblock Gradient methods for minimizing composite functions.
\newblock \emph{Mathematical Programming}, 140\penalty0 (1):\penalty0 125--161,
  2013.

\bibitem[Rockafellar(1976)]{rockafellar1976augmented}
Rockafellar, R.~T.
\newblock Augmented lagrangians and applications of the proximal point
  algorithm in convex programming.
\newblock \emph{Mathematics of operations research}, 1\penalty0 (2):\penalty0
  97--116, 1976.

\bibitem[Rockafellar(2015)]{rockafellar2015convex}
Rockafellar, R.~T.
\newblock \emph{Convex analysis}.
\newblock Princeton university press, 2015.

\bibitem[Shi et~al.(2016)Shi, Zhang, Li, et~al.]{shi2016regression}
Shi, P., Zhang, A., Li, H., et~al.
\newblock Regression analysis for microbiome compositional data.
\newblock \emph{The Annals of Applied Statistics}, 10\penalty0 (2):\penalty0
  1019--1040, 2016.

\bibitem[Sun \& Sun(2002)Sun and Sun]{sun2002semismooth}
Sun, D. and Sun, J.
\newblock Semismooth matrix-valued functions.
\newblock \emph{Mathematics of Operations Research}, 27\penalty0 (1):\penalty0
  150--169, 2002.

\bibitem[Tibshirani(1996)]{tibshirani1996regression}
Tibshirani, R.
\newblock Regression shrinkage and selection via the lasso.
\newblock \emph{Journal of the Royal Statistical Society. Series B
  (Methodological)}, pp.\  267--288, 1996.

\bibitem[Tibshirani(2011)]{tibshirani2011solution}
Tibshirani, R.~J.
\newblock \emph{The solution path of the generalized lasso}.
\newblock Stanford University, 2011.

\bibitem[Xu(2017)]{xu2017accelerated}
Xu, Y.
\newblock Accelerated first-order primal-dual proximal methods for linearly
  constrained composite convex programming.
\newblock \emph{SIAM Journal on Optimization}, 27\penalty0 (3):\penalty0
  1459--1484, 2017.

\bibitem[Yang \& Yuan(2013)Yang and Yuan]{yang2013linearized}
Yang, J. and Yuan, X.
\newblock Linearized augmented lagrangian and alternating direction methods for
  nuclear norm minimization.
\newblock \emph{Mathematics of computation}, 82\penalty0 (281):\penalty0
  301--329, 2013.

\bibitem[Yuan et~al.(2018)Yuan, Sun, and Toh]{pmlr-v80-yuan18a}
Yuan, Y., Sun, D., and Toh, K.-C.
\newblock An efficient semismooth {N}ewton based algorithm for convex
  clustering.
\newblock In Dy, J. and Krause, A. (eds.), \emph{Proceedings of the 35th
  International Conference on Machine Learning}, volume~80 of \emph{Proceedings
  of Machine Learning Research}, pp.\  5718--5726, Stockholmsmässan, Stockholm
  Sweden, 10--15 Jul 2018. PMLR.

\bibitem[Zhao et~al.(2010)Zhao, Sun, and Toh]{zhao2010newton}
Zhao, X.-Y., Sun, D., and Toh, K.-C.
\newblock A newton-cg augmented lagrangian method for semidefinite programming.
\newblock \emph{SIAM Journal on Optimization}, 20\penalty0 (4):\penalty0
  1737--1765, 2010.

\bibitem[Zhou \& So(2017)Zhou and So]{zhou2017unified}
Zhou, Z. and So, A. M.-C.
\newblock A unified approach to error bounds for structured convex optimization
  problems.
\newblock \emph{Mathematical Programming}, 165\penalty0 (2):\penalty0 689--728,
  2017.

\bibitem[Zou(2006)]{zou2006adaptive}
Zou, H.
\newblock The adaptive lasso and its oracle properties.
\newblock \emph{Journal of the American statistical association}, 101\penalty0
  (476):\penalty0 1418--1429, 2006.

\bibitem[Zou \& Hastie(2005)Zou and Hastie]{zou2005regularization}
Zou, H. and Hastie, T.
\newblock Regularization and variable selection via the elastic net.
\newblock \emph{Journal of the Royal Statistical Society: Series B (Statistical
  Methodology)}, 67\penalty0 (2):\penalty0 301--320, 2005.

\end{thebibliography}
\bibliographystyle{icml2019}

\onecolumn
\appendix
\section{Proofs}
\subsection{Proofs for convergence of augmented Lagrangian method }
Recall the definition of metric subregularity.
\begin{define}
	A multi-valued mapping $F:\mathcal{X}\rightrightarrows\mathcal{Y}$ is said to be metrically subregular at $\bar{x}\in\mathcal{X}$ for $\bar{y}\in\mathcal{Y}$ with modulus $\kappa\geq0$ where $(\bar{x},\bar{y})\in\text{gph}(F)$, if there exist neighborhoods $\mathcal{E}_1$ of $\bar{x}$ and $\mathcal{E}_2$ of $\bar{y}$ such that
	\begin{align*}
	\text{dist}(x,F^{-1}(\bar{y}))\leq\kappa\text{dist}(\bar{y},F(x)\cap\mathcal{E}_2),\qquad\forall x\in\mathcal{E}_1.
	\end{align*}
\end{define}
We give the following proposition from \cite{aragon2008characterization} to provide a relatively easy way to check the metrically subregularity of closed proper convex function.
\begin{prop}
	Denote $\mathcal{H}$ as the Hilbert space. Let $f:\mathcal{H}\rightarrow(-\infty,+\infty]$ be a proper lower semicontinuous convex function and $(\bar{x},\bar{s})\in\text{gph}(\partial f)$. Then $\partial f$ is metrically subregular at $\bar{x}$ for $\bar{s}$ if and only if the following holds for $\forall x\in\mathcal{E}$:
	\begin{align}
	f(x)\geq f(\bar{x})+\langle\bar{s},x-\bar{x}\rangle+\kappa\text{dist}^{2}(x,(\partial f)^{-1}(\bar{s})),
	\end{align}
	where constant $\kappa>0$ and $\mathcal{E}$ is a neighborhood of $\bar{x}$
\end{prop}
Then we give the definition of boundedly linearly regularity for a collection of $m$ closed convex sets $\{C_1,\dots,C_m\}$ from \cite{bauschke1996projection}.
\begin{define}
	For $m$ closed convex sets $\{C_1,\dots,C_m\}$ which belong to $\mathcal{X}$. Suppose the intersection $C:=C_1\cap C_2\cap\dots\cap C_m$ is non-empty. We say that collection $\{C_1,\dots,C_m\}$ is boundedly  linearly regular if for every bounded set $\mathcal{S}\subseteq\mathcal{X}$, there exist a constant $\kappa>0$ such that for $\forall x\in\mathcal{S}$,
	\begin{align*}
	\text{dist}(x,C)\leq\kappa\max\{\text{dist}(x,C_1),\dots,\text{dist}(x,C_m)\}.
	\end{align*}
\end{define}
From corollary 3 in \cite{bauschke1999strong}, we have the following sufficient condition to guarantee the boundedly linearly regularity.
\begin{prop}
	Let $C_1,\dots,C_m$ be closed convex sets in $\mathcal{X}$. Suppose $C_{r+1},\dots,C_m$ are polyhedral for some $r\in\{0,\dots,m-1\}$. We say that the collection $\{C_1,\dots,C_m\}$ is boundedly linearly regularity if 
	\begin{align*}
	\bigcap_{i=1}^{r}\text{ri}(C_i)\cap\bigcap_{i=r+1}^{m}C_{i}\neq\emptyset.
	\end{align*}
\end{prop}
Here we give following lemma to show the invariant property of $Ax$ over $\mathcal{T}^{-1}_{\phi}(0)$, the proof follows from \cite{luo1992linear} and we omit here.
\begin{lem}
	$Ax$ is invariant over $x\in\mathcal{T}_{\phi}^{-1}(0)$, which means that if $x,x'\in\mathcal{T}_{\phi}^{-1}(0)$, then $Ax=Ax'$.
\end{lem}
Combine above lemma, recall the primal problem (P) and KKT system (\ref{A.2}) as follows,
\begin{align}\label{A.2}
\begin{split}
\begin{cases}
0\in\partial h^{*}(u)-Ax, \\
0\in\partial p^{*}(w)-x, \\
0 = Bx-d, \\
0= A^Tu-B^{T}v+w,
\end{cases} \quad(u,v,w;x)\in\mathcal{Z}\times\mathcal{X},
\end{split}
\end{align}
it is easy to check that the following observation holds.
\begin{prop}
	Let $(\bar{u},\bar{v},\bar{w},\bar{x})$ be a solution to KKT system (\ref{A.2}), then we can represent the optimal solution set $\mathcal{T}^{-1}_{\phi}(0)$ of (P) as
	\begin{align*}
	\ &\mathcal{T}_{\phi}^{-1}(0) \\
	=\ &\{x\in\mathcal{X}|Ax=\bar{\xi},0=Bx-d,0\in\bar{\eta}+\partial p(x)-B^{T}\bar{v}\}\\ 
	=\ &\mathcal{D}_{1}\cap\mathcal{D}_2\cap\mathcal{D}_3,
	\end{align*}
	where $\bar{\xi}:=A\bar{x}$, $\bar{\eta}=A^{T}\nabla h(\bar{\xi})$, $\mathcal{D}_{1}=\{x\in\mathcal{X}|Ax=\bar{\xi}\}$, $\mathcal{D}_{2}=\{x\in\mathcal{X}|Bx=d\}$ and $\mathcal{D}_{3}=\{x\in\mathcal{X}|0\in\bar{\eta}+\partial p(x)-B^{T}\bar{v}\}$.
\end{prop}
In the next step, for the constrained Lasso model, we have that $h(x)=\frac{1}{2}\|x-b\|^2$ and $p(x)=\lambda\|x\|_1$. Then we give key properties of $h(\cdot)$ and $p(\cdot)$ to help establish the metric subregularity of $\mathcal{T}_{\phi}$.
\begin{prop} \label{prop4}
	The following properties hold:\\
	(a) For any $r\in\text{dom}(h)$, there exists a constant $\kappa_1>0$ and neighborhood $\mathcal{E}_3$ of $r$ such that for $\forall r'\in\mathcal{E}_3$
	\begin{align}\label{A.3}
	h(r')\geq h(r)+\langle \nabla h(r),r'-r\rangle+\kappa_1\|r-r'\|^2,
	\end{align}
	(b) $\partial p$ is metric subregular with constant $\kappa_2>0$, i.e., for any $(x,s)\in\text{gph}(\partial p)$, there exists a constant $\kappa_2>0$ and a neighborhood $\mathcal{E}_4$ of $x$ such that for $\forall x'\in\mathcal{E}_4$
	\begin{align}
	p(x')\geq p(x)+\langle s,x'-x\rangle+\kappa_2\text{dist}^{2}(x',(\partial p)^{-1}(s)).
	\end{align}
\end{prop}
\begin{proof}
	(a) Since $h(r)$ is 1-strongly convex, (\ref{A.3}) holds with $\kappa_1=1$. \\
	(b) $\partial(\|x\|_1)$ is metric subregular by Proposition 11 in \cite{zhou2017unified} since $l_1$ norm of vector is a special case of nuclear norm $\|\cdot\|_{*}$ of matrix. \\
\end{proof}
For problem (P), assume there exists at least one optimal solution $\bar{x}\in\mathcal{T}^{-1}_{\phi}(0)$. We say that the second order growth condition for (P) holds at $\bar{x}$ w.r.t set $\mathcal{T}^{-1}_{\phi}(0)$ if the following holds for $\forall x\in\mathcal{E}\cap\{x\in\mathcal{X}|Bx=d\}$:
\begin{align}
f(x)\geq f(\bar{x})+\kappa\text{dist}^{2}(x,\mathcal{T}^{-1}_{\phi}(0)),
\end{align}
where $\kappa>0$ and $\mathcal{E}$ is a neighborhood of $\bar{x}$. From Proposition 1 we get that the metrically subregularity of $\mathcal{T}_{\phi}$ holds at $\bar{x}$ if and only if the second order growth condition holds at $\bar{x}$ w.r.t set $\mathcal{T}^{-1}_{\phi}(0)$. Hence, we can just show the second order growth conditions holds to state that $\mathcal{T}_{\phi}$ is metrically subregular. The following lemma is a  key step to prove the second order growth condition.	
\begin{lem} \label{lem2}
	For $\bar{x}\in\mathcal{T}^{-1}_{\phi}(0)$, the following holds
	\begin{align*}
	\text{dist}(x,\mathcal{D}_2)\leq\kappa\|d-Bx\|_{2},\quad\forall x\in\mathcal{E},
	\end{align*}
	where $\kappa>0$ and $\mathcal{E}$ is a neighborhood of $\bar{x}$.
\end{lem}
\begin{proof}
	Define $C_1=\{(x,r)\in\mathcal{X}\times\mathcal{V}|Bx-d=r\}$ and $C_2=\{(x,r)\in\mathcal{X}\times\mathcal{V}|r=0\}$. From Proposition 2, we know that $C_1$ and $C_2$ are boundedly linearly regular since both $C_1$ and $C_2$ are polyhedral and the intersection $\bar{\mathcal{D}}_2=C_1\cap C_2$ has an element $(\bar{x},d-B\bar{x})$ which is non-empty. Hence, we have that there exists a constant $\kappa>0$ and a neighborhood of $\mathcal{E}$ of $\bar{x}$ such that 
	\begin{align*}
	\text{dist}((x,d-Bx),\bar{\mathcal{D}}_2)\leq\kappa(\text{dist}((x,d-Bx),C_1)+\text{dist}((x,d-Bx),C_2))=\kappa\|d-Bx\|_2
	\end{align*}
	where the equality holds since $\text{dist}((x,d-Bx),C_1)=0$. Moreover, it's easy to see that there exists $(x',v')\in\bar{\mathcal{D}}_2$ such that
	\begin{align*}
	\text{dist}((x,d-Bx),\bar{\mathcal{D}}_2)=\sqrt{\|x-x'\|^2+\|d-Bx-v'\|^2}\geq\|x-x'\|_2=\text{dist}(x,\mathcal{D}_2).
	\end{align*}
	Hence we complete the proof. \\
\end{proof}
Based on above results, now we claim that the metrically subregularity of $\mathcal{T}_\phi$ holds at $\bar{x}$ for our problem in Theorem \ref{thr1} where $\bar{x}$ is the optimal solution to (P).
\begin{thr} \label{thr1}
	Assume that $\mathcal{T}^{-1}_l(0)$ is non-empty and there exists $(\bar{u},\bar{v},\bar{w})\in\mathcal{T}_\psi^{-1}(0)$. For $h(\cdot)$ and $p(\cdot)$ chosen as in constrained Lasso model. Then metrically subregularity of $\mathcal{T}_\phi$ holds at $\bar{x}$.
\end{thr} 
\begin{proof}
	For (P), we can rewrite it as 
	\begin{align*}
	\min_{x}\delta_{0}(d-Bx)+h(Ax)+p(x),
	\end{align*}
	where $\delta_{0}(\cdot)$ is an indicator function, i.e. $\delta_{0}(e)=0$ if $e=0$ and $\delta_{0}(e)=\infty$ if $e\neq 0$. Let $\bar{x}\in\mathcal{T}^{-1}_{\phi}(0)$. It is easy to see that the following holds
	\begin{align}
	\delta_{0}(d-Bx)\geq\delta_{0}(d-B\bar{x})+\langle\bar{v},d-Bx-(d-B\bar{x})+\kappa_3\text{dist}^2(d-Bx,0),\quad\forall x\in\mathcal{E},
	\end{align}
	where $\kappa_3>0$ and $\mathcal{E}$ is a neighborhood of $\bar{x}$. Since there exists $(\bar{u},\bar{v},\bar{w})\in\mathcal{T}_\psi^{-1}(0)$, we get that $\{\mathcal{D}_1,\mathcal{D}_2,\mathcal{D}_3\}$ are boundedly linearly rugular and $\mathcal{T}^{-1}_{\phi}(0)=\mathcal{D}_1\cap\mathcal{D}_2\cap\mathcal{D}_3$, by definition we get that 
	\begin{align}
	\begin{split}
	\text{dist}^{2}(x,\mathcal{T}^{-1}_{\phi}(0))\ &=\text{dist}^2(x,\mathcal{D}_1\cap\mathcal{D}_2\cap\mathcal{D}_3) \\
	\ &\leq\kappa_4(\text{dist}^2(x,\mathcal{D}_1)+\text{dist}^2(x,\mathcal{D}_2)+\text{dist}^2(x,\mathcal{D}_3)) \\
	\ &\leq\kappa_5(\|Ax-\bar{\xi}\|^2+\|d-Bx\|^2+\text{dist}^2(x,(\partial p)^{-1}(B^{T}v-\bar{\eta}))),
	\end{split}
	\end{align}
	where $\kappa_4,\kappa_5>0$ and $x\in\mathcal{E}$, the first term of last inequality comes from Hoffman's error bound and second term comes from Lemma \ref{lem2}. Since we choose $h(x)=\frac{1}{2}\|x-b\|^2$ and $p(x)=\lambda\|x\|_{1}$, combine Proposition \ref{prop4} and above results, we can get that for any $x\in\mathcal{E}\cap\{x\in\mathcal{X}|Bx=d\}$,
	\begin{align*}
	f(x) =\ & f(x)+\delta_{0}(d-Bx) \\
	= \ &h(Ax)+p(x)+\delta_{0}(d-Bx) \\
	\geq\ & h(\bar{\xi})+\langle\nabla h(\bar{\xi}),Ax-\bar{\xi}\rangle+\kappa_1\|Ax-\bar{\xi}\|^2 +p(\bar{x})+\langle B^{T}\bar{v}-\bar{\eta},x-\bar{x}\rangle \\
	\ &\kappa_2\text{dist}^2(x,(\partial p)^{-1}(B^{T}\bar{v}-\bar{\eta}))+\delta_{0}(d-B\bar{x})+\langle\bar{v},d-Bx-(d-B\bar{x})+\kappa_3\|d-Bx\|^2 \\
	=\ &f(\bar{x})+\langle A^{T}\nabla h(\bar{\xi})-\bar{\eta},x-\bar{x}\rangle+(\kappa_1\|Ax-\bar{\xi}\|^2+\kappa_2\text{dist}^2(x,(\partial p)^{-1}(B^{T}\bar{v}-\bar{\eta}))+\kappa_3\|d-Bx\|^2) \\
	\geq \ & f(\bar{x})+\min(\kappa_1,\kappa_2,\kappa_3)(\|Ax-\bar{\xi}\|^2+\text{dist}^2(x,(\partial p)^{-1}(B^{T}\bar{v}-\bar{\eta}))+\|d-Bx\|^2) \\
	\geq \ &f(\bar{x})+\dfrac{\min(\kappa_1,\kappa_2,\kappa_3)}{\kappa_5}\text{dist}^{2}(x,\mathcal{T}^{-1}_{\phi}(0)).
	\end{align*}
	Therefore, we prove the second order growth condition holds for (P) at $\bar{x}$ w.r.t $\mathcal{T}^{-1}_{\phi}(0)$. As a consequence,  the metrically subregularity of $\mathcal{T}_\phi$ holds at $\bar{x}$.\\
\end{proof}

\subsection{Proofs for convergence of semismooth Newton method}
Recall that we use semismooth Newton method to solve the following Newton linear system:
\begin{align}\label{A.8}
\nabla\theta(y)=0,\qquad \forall y\in\text{dom}(y),
\end{align}
where $y=(u,v)$.

We prove the local superlinear convergence rate of our semismooth Newton method based on the following propositions. 
\begin{prop} \label{prop5}
	Let $Q\in\mathbb{S}^n$ be a positive semidefinite matrix and $\sigma>0$. Then $BQB^{T}$ is positive definite if and only if 
	\begin{align} \label{4.12}
	\left\langle\begin{bmatrix}
	u \\ v
	\end{bmatrix},\left(\begin{bmatrix}
	\textbf{I}_{m} &  \\
	& \textbf{0}
	\end{bmatrix}+\sigma\begin{bmatrix}
	A \\ -B
	\end{bmatrix}Q\begin{bmatrix}
	A^{T} &-B^{T}
	\end{bmatrix}\right)\begin{bmatrix}
	u \\ v
	\end{bmatrix}\right\rangle >0,
	\end{align}
	for all $(u,v)\in\text{dom}(y)\backslash\{(0,0)\}$.
\end{prop}
\begin{proof}
	For "$\Leftarrow$" case, since (\ref{4.12}) holds for all $(u,v)\in\text{dom}(y)\backslash\{(0,0)\}$, we can always choose $(u,v)=(0,v')$ where arbitrary $v'\in\mathbb{R}^s\backslash\{0\}$, hence $(v')^T(BQB^{T})v'>0$ for all $v'\in\mathbb{R}^s\backslash\{0\}$. Thus $BQB^{T}$ is positive definite.  \\
	For "$\Rightarrow$" case, suppose $BQB^{T}$ is positive definite, we proof it by contradictory. We assume there exists $(u',v')\neq(0,0)$, such that for (\ref{4.12}), we have
	\begin{align}\label{4.13}
	(u')^{T}u'+\sigma\left\langle\begin{bmatrix}
	u' \\ v'
	\end{bmatrix},\left(\begin{bmatrix}
	A \\ -B
	\end{bmatrix}Q\begin{bmatrix}
	A^{T} &-B^{T}
	\end{bmatrix}\right)\begin{bmatrix}
	u' \\ v'
	\end{bmatrix}\right\rangle=0.
	\end{align}
	Since $Q$ is positive semidefinite, we know that $\begin{bmatrix}
	A \\ -B
	\end{bmatrix}Q\begin{bmatrix}
	A^{T} &-B^{T}
	\end{bmatrix}$ is also positive semidefinite. So both first term and second term in (\ref{4.13}) are nonnegative, we must have both term to be zero, which means $u'=0$ and moreover $(v')^T(BQB^{T})v'=0$. Since $BQB^T$ is positive definite, we have $v'=0$. Hence $(u',v')=(0,0)$ which makes the contradiction. This completes the proof of proposition. \\	
\end{proof}
Recall that
$Q\in \partial\text{Prox}_{\sigma\lambda\|x\|_{1}}(x)$ is $Q=\text{diag}(q)$, a diagonal matrix with $i$-th element chosen as 
\begin{align}\label{A.11}
q_{i}=\begin{cases}
1,\qquad\text{if}\quad|x_{i}|>\sigma\lambda, \\
0,\qquad\text{otherwise},
\end{cases}
\end{align}
and set $\mathcal{J}=\{j:|x_{j}|>\sigma\lambda\}$definition of $Q, \mathcal{J}$ and $r$, we can also ensure the positive definiteness of Newton linear system (\ref{A.8}) under the \textit{constraint nondegeneracy condition} in the following proposition.
\begin{prop} \label{prop6}
	Let $(\hat{u},\hat{v})$ be the solution of Newton linear system (\ref{A.8}). Let $\hat{z}=x-\sigma(A^T\hat{u}-B^{T}\hat{v})$ and $\hat{Q}=\text{diag}(\hat{q})\in\partial\text{Prox}_{\sigma\lambda\|x\|_1}(\hat{z})$ with $\hat{q}$ is defined as in (\ref{A.11}). If the constraint nondegeneracy condition, i.e.,
	\begin{align}
	\text{lin}(B\hat{Q})=\mathbb{R}^s
	\end{align} 
	holds at $\hat{z}$ with $\text{lin}(B\hat{Q})$ means the linear space of $B\hat{Q}$. Then 
	\begin{align*}
	\begin{bmatrix}
	\textbf{I}_{d} &  \\
	& \textbf{0}
	\end{bmatrix}+\sigma\begin{bmatrix}
	A \\ -B
	\end{bmatrix}\hat{Q}\begin{bmatrix}
	A^{T} &-B^{T}
	\end{bmatrix}
	\end{align*}
	is positive definite on $\text{dom}(y)$.
\end{prop}
\begin{proof}
	By the definition of $\hat{Q}$ we know that it is a idempotent matrix which means that $\hat{Q}^2=\hat{Q}$. The constraint nondegeneracy conditions implies that $B\hat{Q}B^T=(B\hat{Q})(B\hat{Q})^T$ is positive definite. Then by Proposition \ref{prop5}, we get the desired results. \\
\end{proof}
Now we can establish the superlinear convergence of Algorithm 2.
\begin{thr} \label{thr5}
	Let $\{\hat{u},\hat{v}\}$ be the accumulation point of the sequence $\{(u^{j},v^{j})\}$ generated by Algorithm 2. Assume the constraint nondegeneracy condition holds at $\hat{z}=x-\sigma(A^T\hat{u}-B^{T}\hat{v})$. Then the sequence  $\{(u^j,v^j)\}$ converges to $\{(\hat{u},\hat{v})\}$ and 
	\begin{align*}
	\|(u^{j+1},v^{j+1})-(\hat{u},\hat{v})\|=\mathcal{O}(	\|(u^{j},v^{j})-(\hat{u},\hat{v})\|^{1+\tau}).
	\end{align*}
\end{thr}
\begin{proof}
	Since constraint nondegenerate condition holds, by Proposition \ref{prop6}, $V^{*}\in\hat{\partial}^2\theta(y^{*})$ is positive definite. Hence we can obtain the superlinear convergence result from Theorem 3.5 in \cite{zhao2010newton}.\\
\end{proof}
For a special scenario that sum to zero constraint which means $B=e^{T}$, we have the following corollary.
\begin{cor} \label{cor1}
	For $B=e^{T}$, which is the sum to zero constraint. Assume the optimal solution to Newton linear system is non-zero. Let $\{\hat{u},\hat{v}\}$ be the accumulation point of the sequence $\{(u^{j},v^{j})\}$ generated by Algorithm 2. Then the sequence  $\{(u^j,v^j)\}$ converges to $\{(\hat{u},\hat{v})\}$ with superlinear convergence.
\end{cor}
\begin{proof}
	Since $B=e^{T}$, which means $B$ is full row rank with $\text{rank}(B)=1$. Moreover, when optimal solution to Newton linear system is non-zero, we get that $r\geq 1=s$ and constraint nondegeneracy condition holds. Therefore we know our result holds from Theorem \ref{thr5}, . \\
\end{proof}

\section{Numerical Experiments on Synthetic Data}
For the iteration numbers in our tables, $(p|q)$ denote the outer loop $p$ and total inner iterations $q$, $(-)$ means that the algorithm achieves maximum iteration number $10000$. Moreover, we denote our algorithm as SSNAL. Running time counts in seconds.
\subsection{Sum to zero constraints}
In this scenario, we set $B=e^{T}$ and $d=0$.
\begin{table}[!hbp] 
	\small
	\caption{\small Performance of SSNAL, primal dual method, linearized ALM, ADMM and A-ADMM with sum to zero constraints on synthetic data sets. $d$ is the sample size and $n$ is the dimension of each sample. $\lambda_l$ controls the penalty parameter in (P). 'nnz' denotes the number of nonzeros of the solution obtained by our algorithm. 'opt' is the optimal function value of (P). $\eta_{\text{gap}}$ is the optimal gap. 'a'= our algorithm, 'b' = primal dual method, 'c'= linearized alm, 'd' = ADMM, 'e' = A-ADMM. Running time counts in seconds. }	
	\label{tab1}
	\begin{tabular}{|c|c|c|c|c|c|c|c|c|c|c|c|c|c|}
		\hline  & $\lambda_l$& nnz & opt &\multicolumn{5}{|c|}{ optimal gap, $\eta_{\text{gap}}$} &\multicolumn{5}{|c|}{ running time (iteration number) }   \\
		\hline size $m;n$ & & & & a& b& c& d& e&a&b&c&d&e \\
		\hline\multirow{3}{*}{200;2000} & $10^{-2}$& 189 & 1.3689+4& 3.2-8&6.1+0&2.0+1&2.0-2&1.5-2&\textbf{0.4($10|100$)}&5.3(-)&5.0(-)&6.2(-)&8.3(-)\\
		& $10^{-3}$&193&1.3843+3&3.4-8&5.5+1&8.8+1&1.6-3&1.8-3
		&\textbf{1.0($16|180$)}&6.5(-)&4.6(-)&5.9(-)&7.9(-)\\
		&$10^{-4}$&195&1.3859+2&1.9-9&3.5+1&4.5+2&1.4-4&1.6-4&\textbf{1.3($17|190$)}&6.3(-)&6.7(-)&7.3(-)&6.4(-)   \\
		\hline\multirow{3}{*}{300;3000} & $10^{-2}$& 286 & 3.8649+4& 6.6-8&1.5+1&3.7+1&4.7-1&4.7-1&\textbf{1.2($11|110$)}&19(-)&18(-)&22(-)&25(-)\\
		& $10^{-3}$&289&3.9137+3&1.2-9&1.0+2&1.7+2&1.1-2&1.3-2
		&\textbf{2.3($13|150$)}&18(-)&17(-)&23(-)&22(-)\\
		&$10^{-4}$&291&3.9186+2&9.8-9&8.5+1&1.1+2&1.5-3&1.5-3&\textbf{3.2($16|193$)}&18(-)&25(-)&24(-)&24(-)   \\
		\hline\multirow{3}{*}{500;5000} & $10^{-2}$&471&1.2391+5&2.6-7&2.1+1&7.5+1&7.0+0&7.0+0&\textbf{2.7($11|106$)}&76(-)&80(-)&155(-)&162(-) \\
		& $10^{-3}$&481&1.2555+4&4.7-8&3.3+2&5.6+2&2.6-1&2.6-1&\textbf{3.8($13|153$)}&70(-)&95(-)&158(-)&158(-) \\
		&$10^{-4}$&484&1.2572+3&4.5-9&2.7+2&3.5+2&4.5-3&4.9-3&\textbf{5.3($14|175$)}&78(-)&81(-)&155(-)&159(-)\\
		\hline
		\multirow{3}{*}{800;8000} & $10^{-2}$&746&3.1162+5&1.5-5&8.4+1&3.0+2&7.4+1&7.4+1&\textbf{6.2($10|110$)}&151(-)&149(-)&355(-)&359(-)\\
		& $10^{-3}$&760&3.1602+4&1.2-6&7.7+2&1.3+3&6.4-1&6.2-1&\textbf{10($13|175$)}&153(-)&159(-)&357(-)&361(-)\\
		&$10^{-4}$& 764&3.1647+3&5.4-9&6.4+2&8.4+2&3.0-2&3.4-2&\textbf{15($19|231$)}&157(-)&169(-)&355(-)&363(-) \\
		\hline
		\multirow{3}{*}{1000;10000} & $10^{-2}$&935&5.4143+5&6.4-6&1.5+2&4.8+2&3.3+2&3.3+2&\textbf{9.0($10|101$)}&225(-)&249(-)&551(-)&553(-)\\
		& $10^{-3}$&958&5.4901+4&3.3-7&1.2+3&2.2+3&1.5+0&1.5+0&\textbf{18($15|200$)}&252(-)&263(-)&546(-)&556(-)\\
		&$10^{-4}$&964&5.4978+3&3.3-8&1.1+3&1.5+3&4.5-2&4.5-2&\textbf{23($17|232$})&236(-)&231(-)&550(-)&556(-)  \\
		\hline		
	\end{tabular}
\end{table}
\begin{figure}[!hbp] 
	\centering
	\begin{minipage}[t]{0.26\linewidth}
		\includegraphics[width=1.0\textwidth]{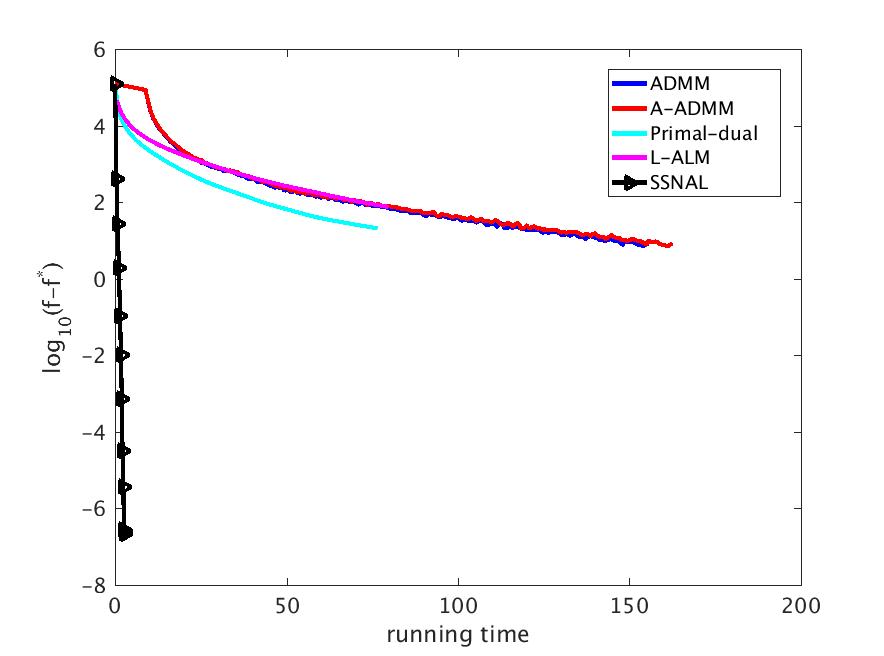}
	\end{minipage}
	\begin{minipage}[t]{0.26\linewidth}
		\includegraphics[width=1.0\textwidth]{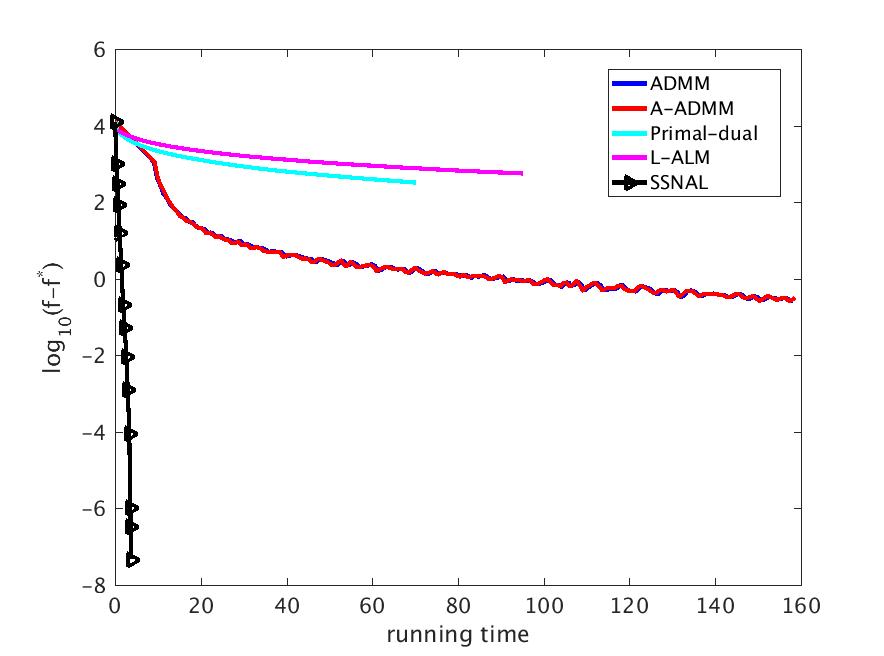}
	\end{minipage}
	\begin{minipage}[t]{0.26\linewidth}
		\includegraphics[width=1.0\textwidth]{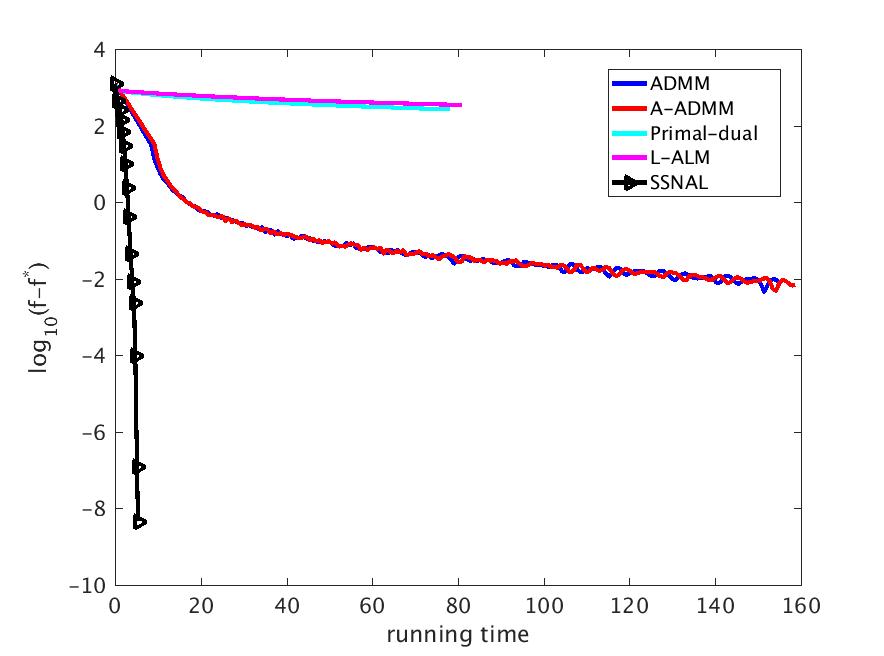}
	\end{minipage} \\	
	\begin{minipage}[t]{0.26\linewidth}
		\includegraphics[width=1.0\textwidth]{Figures//10000_1000_1_224.jpg}
	\end{minipage}
	\begin{minipage}[t]{0.26\linewidth}
		\includegraphics[width=1.0\textwidth]{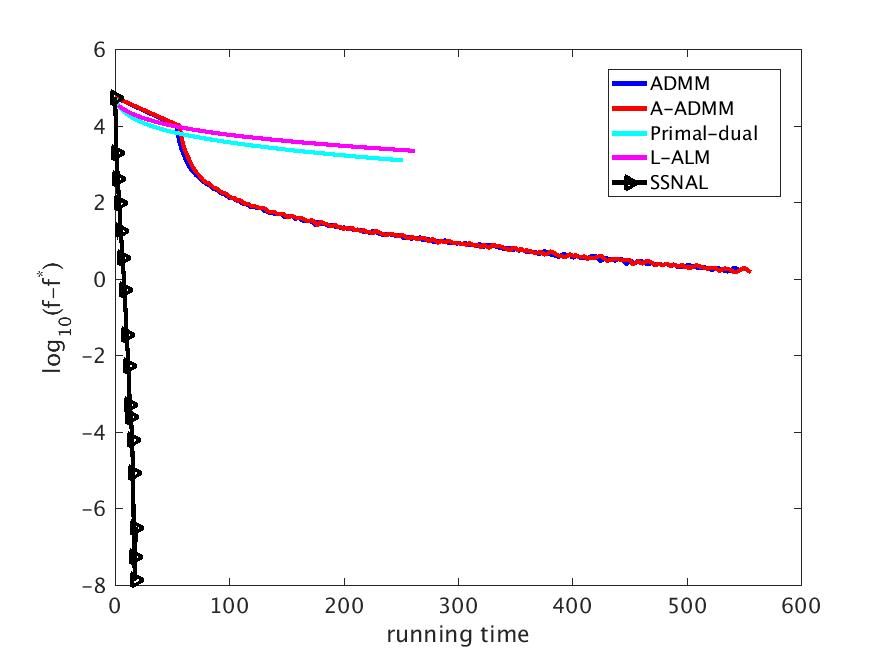}
	\end{minipage}
	\begin{minipage}[t]{0.26\linewidth}
		\includegraphics[width=1.0\textwidth]{Figures//10000_1000_1_2.jpg}
	\end{minipage}
	\caption{Constrained Lasso with sum to zero constraints,$\lambda_l=10^{-2},10^{-3},10^{-4}$. Top row is $m=500, n=5000$, bottom row is $m=1000, n=10000$.}	
\end{figure}

\newpage
\subsection{Randomized $B$}	
In this case, we generate $B\in\mathbb{R}^{s\times n}$ and $d\in\mathbb{R}^{s}$ randomly and set $s=30$. 
\begin{table}[!hbp] 
	\small
	\caption{\small Performance of SSNAL, primal dual method, linearized ALM, ADMM and A-ADMM with randomized generated constraints on synthetic data sets. $d$ is the sample size and $n$ is the dimension of each sample. $\lambda_l$ controls the penalty parameter in (P). 'nnz' denotes the number of nonzeros of the solution obtained by our algorithm. 'opt' is the optimal function value of (P). $\eta_{\text{gap}}$ is the optimal gap. 'a'= our algorithm, 'b' = primal dual method, 'c'= linearized alm, 'd' = ADMM, 'e' = A-ADMM. Running time counts in seconds. }
	\label{tab2}
	\begin{tabular}{|c|c|c|c|c|c|c|c|c|c|c|c|c|c|}
		\hline  & $\lambda_l$& nnz & opt &\multicolumn{5}{|c|}{ optimal gap, $\eta_{\text{gap}}$} &\multicolumn{5}{|c|}{ running time (iteration number) }   \\
		\hline size $m;n$ & & & & a& b& c& d& e&a&b&c&d&e \\
		\hline\multirow{3}{*}{200;2000} & $10^{-2}$&220&1.4046+4&2.1-8&3.6+0&1.9+2&5.8-2&5.3-2&\textbf{1.4($12|114$)}&7.9(-)&15(-)&23(-)&25(-)\\
		& $10^{-3}$&222&1.4203+3&6.7-9&4.2+1&2.3+2&1.3-3&1.4-3&\textbf{1.9($17|157$)}&8.1(-)&17(-)&24(-)&25(-)\\
		&$10^{-4}$&224&1.4218+2&9.9-9&3.4+1&6.8+1&3.0-4&3.0-4&\textbf{2.8($27|192$)}&8.7(-)&15(-)&16(-)&18(-)  \\
		\hline\multirow{3}{*}{300;3000} & $10^{-2}$&308&3.9596+4&1.8-6&1.1+1&3.9+2&4.0-1&4.0-1&\textbf{1.4($11|114$)}&12(-)&26(-)&43(-)&46(-)\\
		& $10^{-3}$&320&4.0082+3&6.1-9&9.7+1&5.4+2&2.5-2&2.4-2&\textbf{1.7($13|136$)}&14(-)&35(-)&43(-)&44(-)\\
		&$10^{-4}$&318&4.0131+2&1.1-8&8.2+1&1.8+2&1.1-3&1.0-3&\textbf{2.5($15|181$)}&14(-)&31(-)&43(-)&44(-)  \\
		\hline\multirow{3}{*}{500;5000} & $10^{-2}$&501&1.2547+5&1.37-6&2.9+1&1.1+3&1.1+3&1.0+3&\textbf{3.2($12|116$)}&76(-)&113(-)&180(-)&180(-)\\
		& $10^{-3}$&512&1.2713+4&6.4-9&2.8+2&1.6+3&2.4-1&2.4-1&\textbf{4.4($13|163$)}&77(-)&117(-)&179(-)&182(-) \\
		&$10^{-4}$&511&1.2729+3&2.0-8&2.5+2&5.6+2&4.9-3&4.6-3&\textbf{7.4($18|227$)}&79(-)&117(-)&180(-)&182(-)\\
		\hline
		\multirow{3}{*}{800;8000} & $10^{-2}$&776&3.1414+5&1.6-4&8.8+1&2.9+3&7.1+1&7.1+1&\textbf{6.5($10|110$)}&172(-)&205(-)&380(-)&385(-)\\
		& $10^{-3}$&800&3.1859+4&5.0-8&7.2+2&4.0+3&7.3-1&7.3-1&\textbf{11($14|175$)}&163(-)&217(-)&388(-)&394(-)\\
		&$10^{-4}$&799&3.1903+3&8.0-8&6.3+2&1.4+3&3.2-2&3.1-2&\textbf{16($18|251$)}&152(-)&227(-)&385(-)&387(-)  \\
		\hline
		\multirow{3}{*}{1000;10000} & $10^{-2}$&974&5.4457+5&1.0-5&1.3+2&4.7+3&2.9+2&2.9+2&\textbf{11($10|117$)}&246(-)&316(-)&602(-)&601(-)\\
		& $10^{-3}$&986&5.5229+4&1.2-6&1.2+3&6.8+3&1.3+0&1.2+0&\textbf{17($15|185$)}&252(-)&321(-)&610(-)&606(-)\\
		&$10^{-4}$&990&5.5307+3&5.2-9&1.1+3&2.5+3&6.8-2&6.8-2&\textbf{22($16|214$)}&258(-)&318(-)&596(-)&598(-)  \\
		\hline		
	\end{tabular}
\end{table}
\begin{figure}[!hbp]
	\centering 
	\begin{minipage}[t]{0.33\linewidth}
		\includegraphics[width=1.0\textwidth]{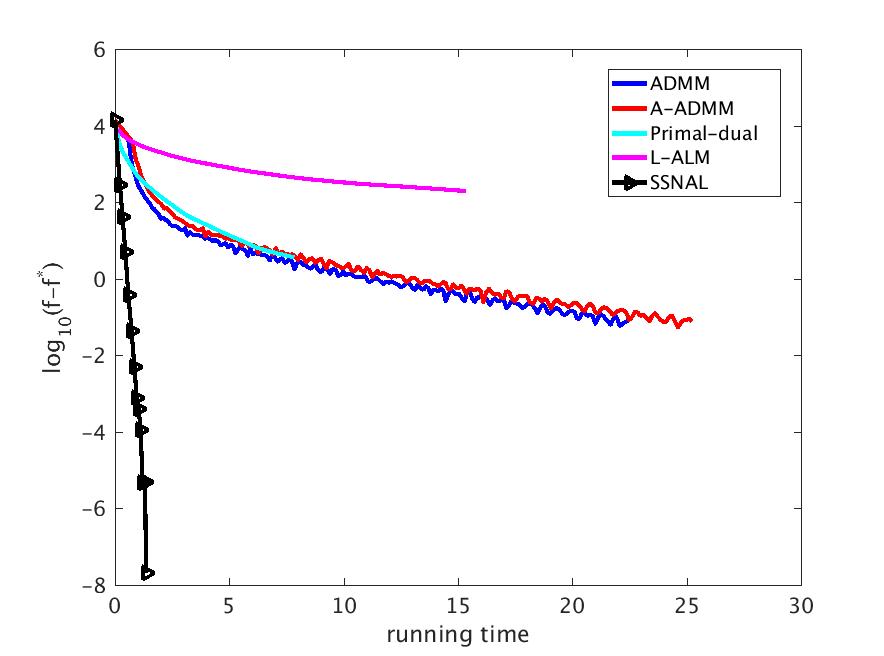}
	\end{minipage}
	\begin{minipage}[t]{0.33\linewidth}
		\includegraphics[width=1.0\textwidth]{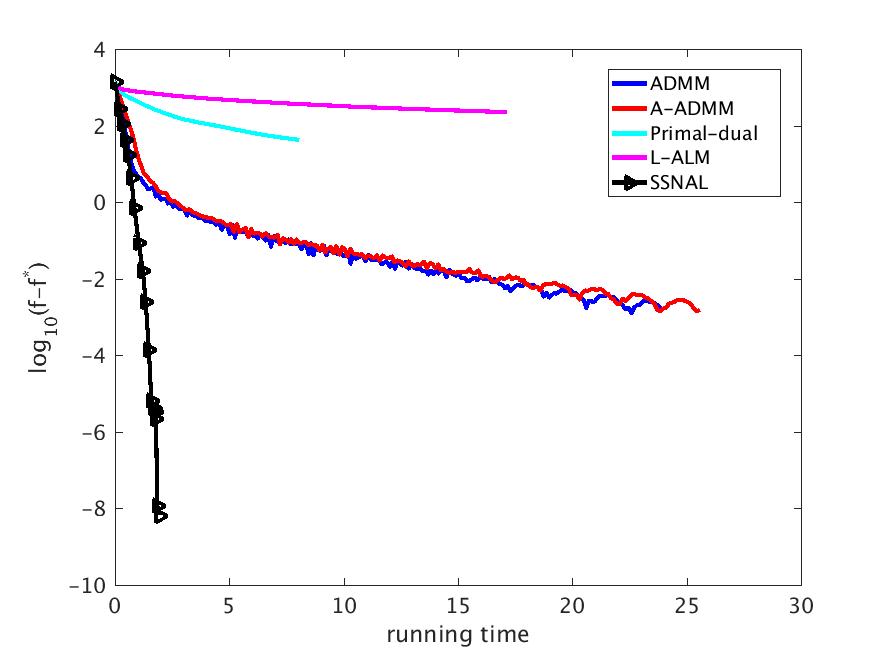}
	\end{minipage}
	\begin{minipage}[t]{0.33\linewidth}
		\includegraphics[width=1.0\textwidth]{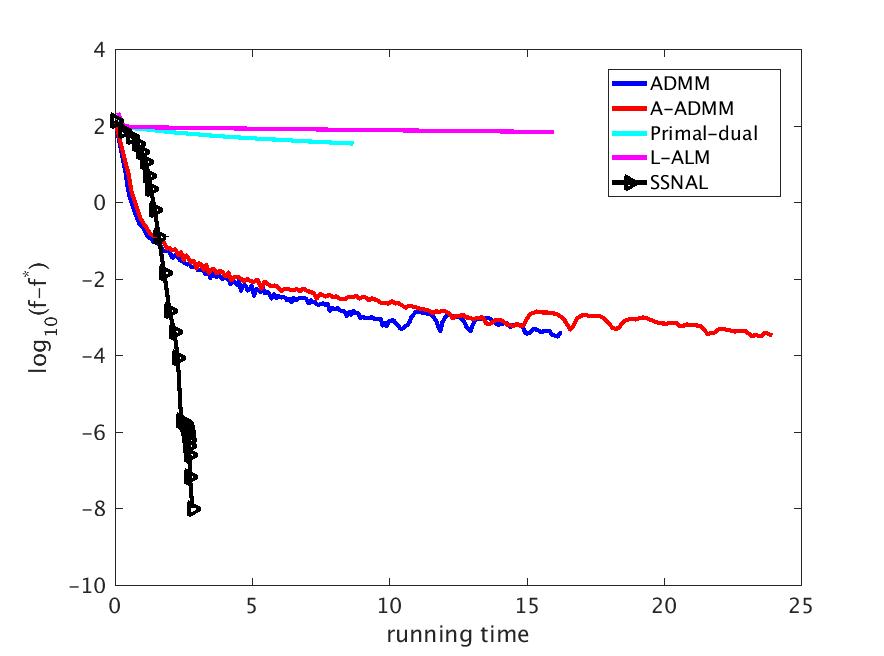}
	\end{minipage} \\
	\begin{minipage}[t]{0.33\linewidth}
		\includegraphics[width=1.0\textwidth]{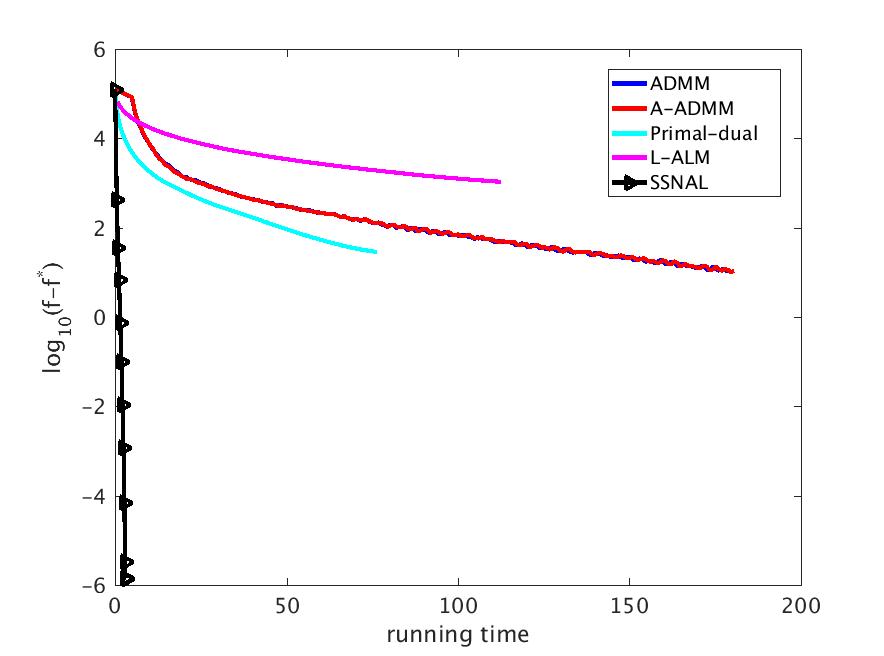}
	\end{minipage}
	\begin{minipage}[t]{0.33\linewidth}
		\includegraphics[width=1.0\textwidth]{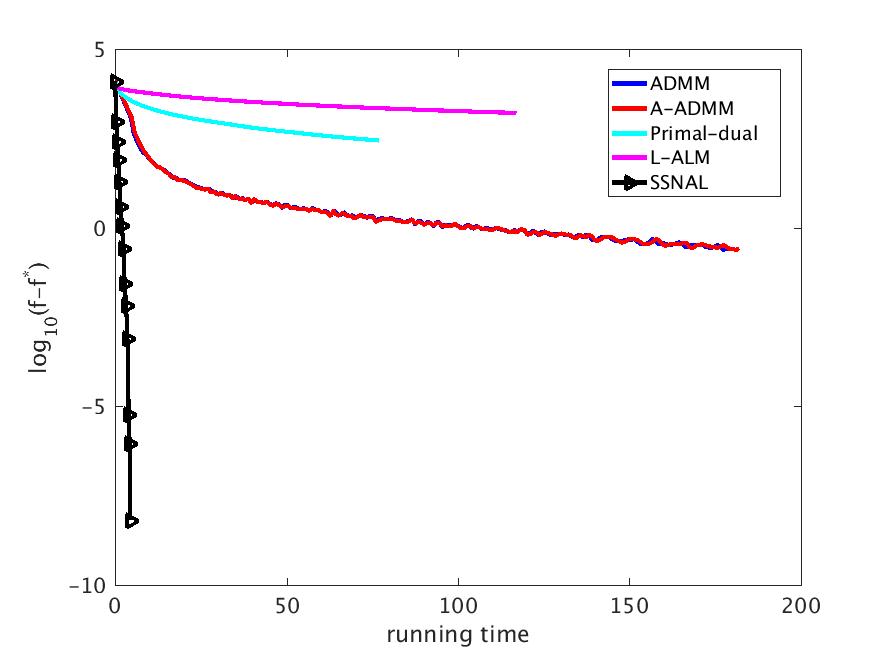}
	\end{minipage}
	\begin{minipage}[t]{0.33\linewidth}
		\includegraphics[width=1.0\textwidth]{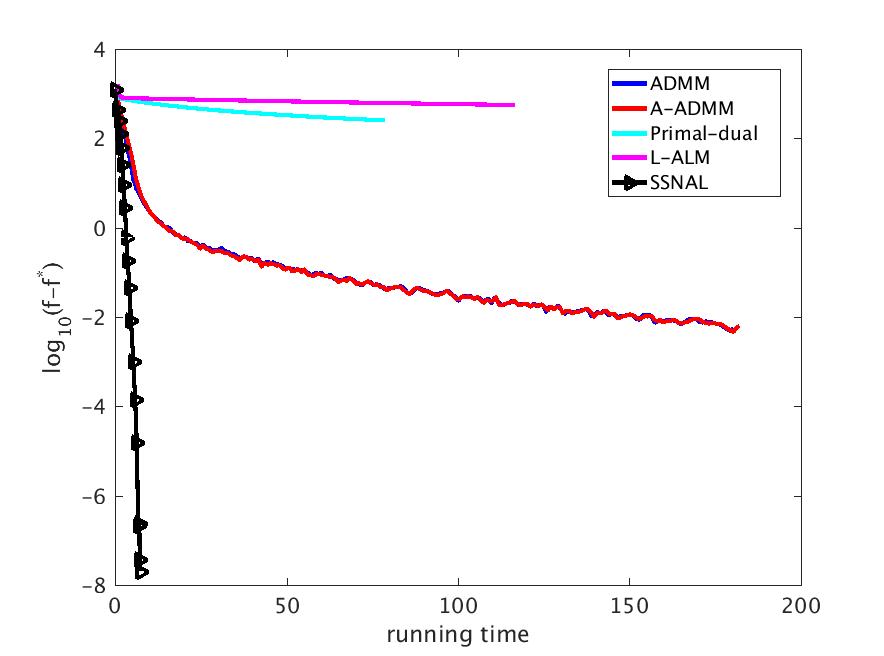}
	\end{minipage}
	\caption{Constrained Lasso with randomized generated constraints, $\lambda_l=10^{-2},10^{-3},10^{-4}$. Top row is $m=200, n=2000, s=30$ and bottom row is $m=500, n=5000, s=30$.}	
\end{figure}

\newpage
\subsection{Generalized Lasso}
In this scenario, we construct $D=\begin{bmatrix}
D_1 \\ D_2
\end{bmatrix}$, where $D_1$ is an $n\times n$ identity matrix and $D_2$ is an $s\times n$ random matrix. Moreover, we set $s=30$. 
\begin{table}[!hbp] 
	\small
	\caption{\small Performance of SSNAL, primal dual method, linearized ALM, ADMM and A-ADMM with generalized Lasso problem on synthetic data sets. $d$ is the sample size and $n$ is the dimension of each sample. $\lambda_l$ controls the penalty parameter in (P). 'nnz' denotes the number of nonzeros of the solution obtained by our algorithm. 'opt' is the optimal function value of (P). $\eta_{\text{gap}}$ is the optimal gap. 'a'= our algorithm, 'b' = primal dual method, 'c'= linearized alm, 'd' = ADMM, 'e' = A-ADMM. Running time counts in seconds. }
	\label{tab3}
	\begin{tabular}{|c|c|c|c|c|c|c|c|c|c|c|c|c|c|}
		\hline  & $\lambda_l$& nnz & opt &\multicolumn{5}{|c|}{ optimal gap, $\eta_{\text{gap}}$} &\multicolumn{5}{|c|}{ running time (iteration number) }   \\
		\hline size $m;n$ & & & & a& b& c& d& e&a&b&c&d&e \\
		\hline\multirow{3}{*}{200;2000} & $10^{-2}$&221&1.4371+4&1.3-6&8.5+0&7.8+0&5.8-2&6.6-2&\textbf{3.7($13|367$)}&8.7(-)&18(-)&25(-)&27(-)\\
		& $10^{-3}$&220&1.4530+3&4.3-8&4.3+1&4.3+1&1.5-3&1.6-3&\textbf{1.9($14|245$)}&8.5(-)&17(-)&23(-)&25(-)\\
		&$10^{-4}$&219&1.4546+2&2.2-8&3.1+1&3.1+1&5.6-4&5.6-4&\textbf{2.2($17|207$)}&10(-)&17(-)&24(-)&24(-)  \\
		\hline\multirow{3}{*}{300;3000} & $10^{-2}$&311&3.8919+4&2.6-6&1.2+1&1.1+1&7.7-1&7.8-1&\textbf{4.6($12|470$)}&14(-)&28(-)&43(-)&43(-)\\
		& $10^{-3}$&317&3.9406+3&4.0-8&8.3+1&8.3+1&1.6-2&1.6-2&\textbf{3.0($14|271$)}&13(-)&30(-)&43(-)&44(-)\\
		&$10^{-4}$&318&3.9455+2&3.9-9&8.0+1&8.1+1&1.0-3&1.0-3&\textbf{2.6($17|208$)}&13(-)&32(-)&41(-)&44(-)  \\
		\hline\multirow{3}{*}{500;5000} & $10^{-2}$&495&1.2179+5&6.8-6&5.1+1&4.7+1&1.2+1&1.1+1&\textbf{12($10|515$)}&74(-)&111(-)&176(-)&179(-)\\
		& $10^{-3}$&505&1.2339+4&6.4-8&3.2+2&3.2+2&7.6-2&8.0-2&\textbf{10($14|397$)}&77(-)&113(-)&180(-)&179(-) \\
		&$10^{-4}$&505&1.2355+3&4.2-8&2.5+2&2.5+2&3.5-3&3.2-3&\textbf{7.5($15|249$)}&85(-)&111(-)&177(-)&179(-)\\
		\hline
		\multirow{3}{*}{800;8000} & $10^{-2}$&779&3.1099+5&6.9-6&1.2+2&1.1+2&8.1+1&8.1+1&\textbf{27($10|525$)}&152(-)&200(-)&383(-)&387(-)\\
		& $10^{-3}$&789&3.1535+4&5.3-7&7.2+1&7.2+1&7.1-1&7.0-1&\textbf{24($14|415$)}&158(-)&216(-)&383(-)&386(-)\\
		&$10^{-4}$&795&3.1578+3&8.5-8&6.2+2&6.2+2&2.8-2&2.7-2&\textbf{17($15|254$)}&152(-)&210(-)&390(-)&392(-)  \\
		\hline
		\multirow{3}{*}{1000;10000} & $10^{-2}$&975&5.2804+5&1.9-5&1.5+2&1.5+2&3.1+2&3.1+2&\textbf{51($11|661$)}&254(-)&311(-)&589(-)&589(-)\\
		& $10^{-3}$&993&5.3519+4&8.2-8&1.2+3&1.2+3&2.7+0&2.8+0&\textbf{41($14|479$)}&234(-)&325(-)&597(-)&595(-)\\
		&$10^{-4}$&995&5.3591+3&7.0-8&1.1+3&1.1+3&7.1-2&7.3-2&\textbf{27($16|282$)}&250(-)&322(-)&589(-)&587(-)  \\
		\hline		
	\end{tabular}
\end{table}
\begin{figure}[!hbp] 
	\centering
	\begin{minipage}[t]{0.33\linewidth}
		\includegraphics[width=1.0\textwidth]{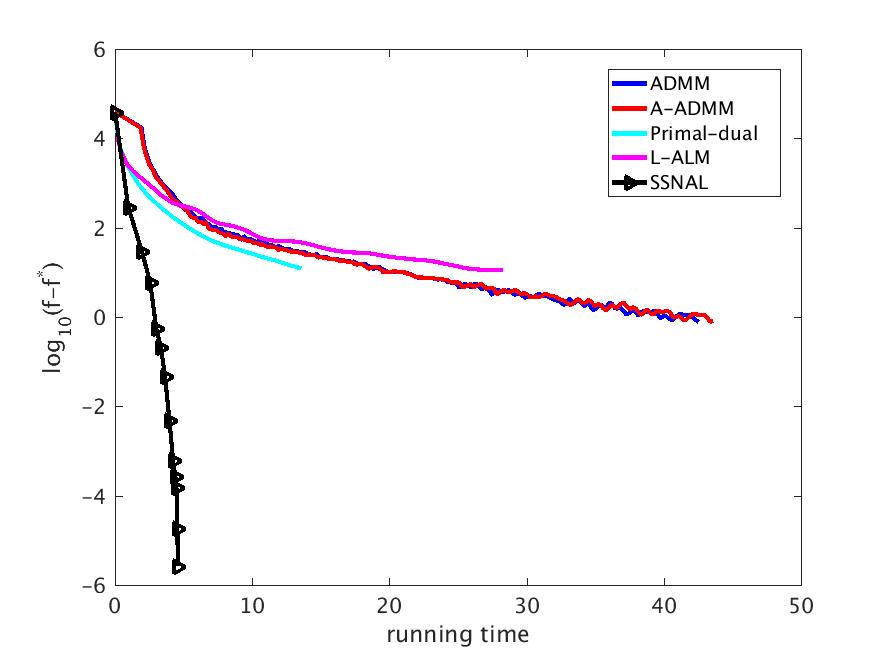}
	\end{minipage}
	\begin{minipage}[t]{0.33\linewidth}
		\includegraphics[width=1.0\textwidth]{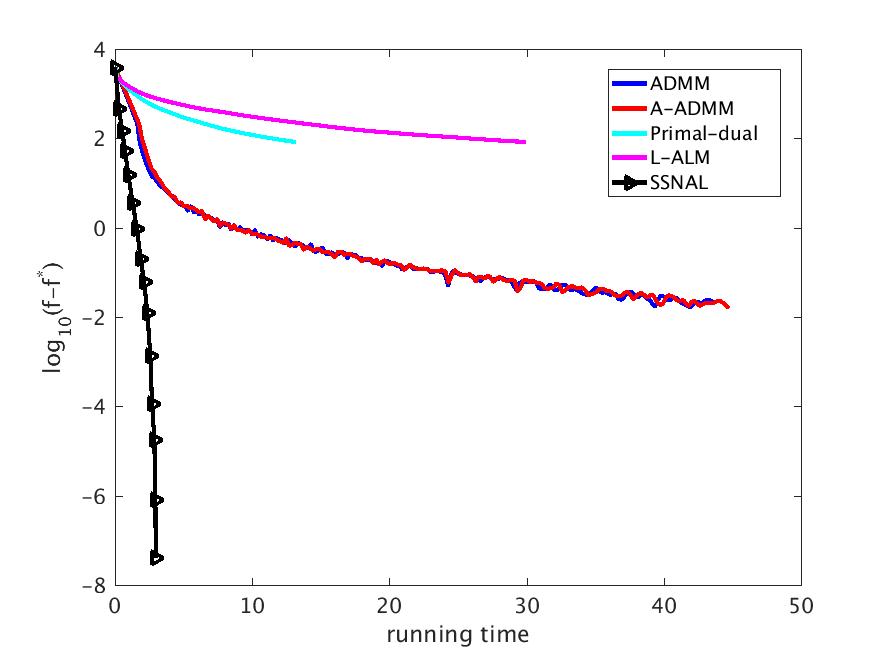}
	\end{minipage}
	\begin{minipage}[t]{0.33\linewidth}
		\includegraphics[width=1.0\textwidth]{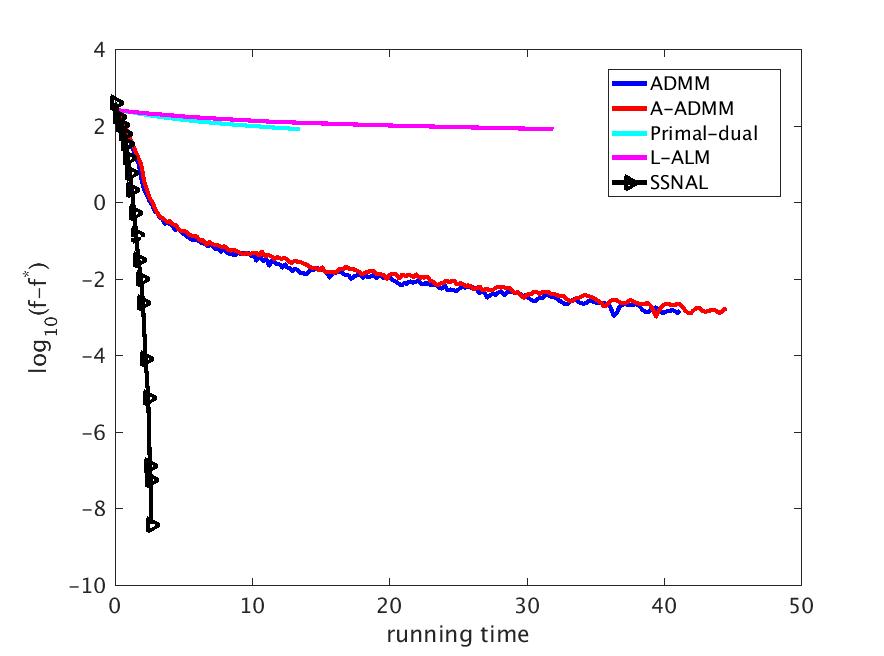}
	\end{minipage} \\
	\begin{minipage}[t]{0.33\linewidth}
		\includegraphics[width=1.0\textwidth]{Figures//8000_800_30_167.jpg}
	\end{minipage}
	\begin{minipage}[t]{0.33\linewidth}
		\includegraphics[width=1.0\textwidth]{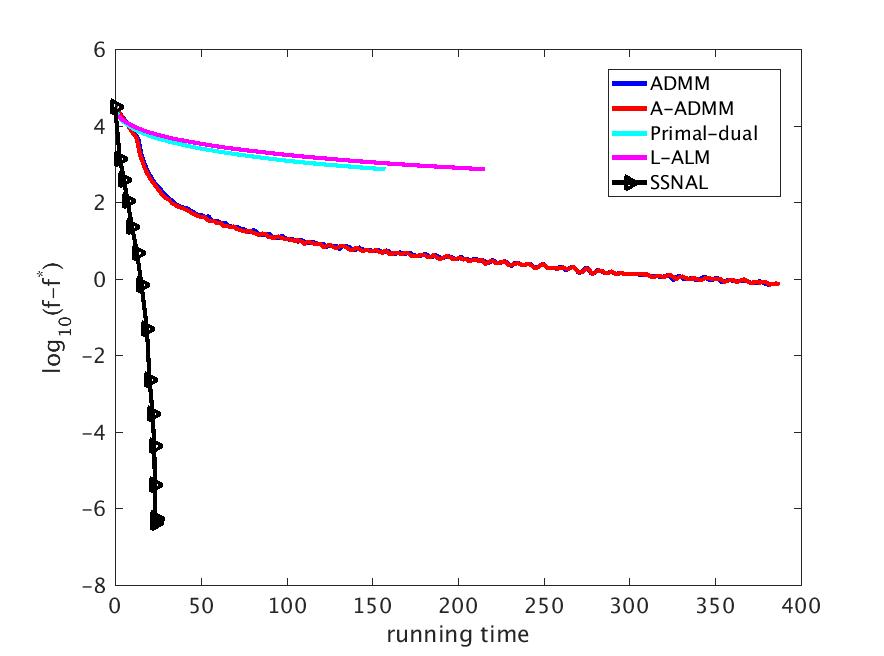}
	\end{minipage}
	\begin{minipage}[t]{0.33\linewidth}
		\includegraphics[width=1.0\textwidth]{Figures//8000_800_30_1.jpg}
	\end{minipage}
	\caption{Constrained Lasso with generalized lasso problem, $\lambda_l=10^{-2},10^{-3},10^{-4}$. Top row is $m=300, n=3000, s=30$ and bottom row is $m=800, n=8000, s=30$.}	
\end{figure}

\newpage
\section{Numerical Results on Real Data}
\subsection{Sum to zero constraint}
In this scenario, we set $B=e^{T}$ and $d=0$.
\begin{table}[!hbp] 
	\scriptsize
	\caption{\small Performance of SSNAL, primal dual method, linearized ALM, ADMM and A-ADMM with sum to zero constraints on UCI regression data sets. $d$ is the sample size and $n$ is the dimension of each sample. $\lambda_l$ controls the penalty parameter in (P). 'nnz' denotes the number of nonzeros of the solution obtained by our algorithm. 'opt' is the optimal function value of (P). $\eta_{\text{gap}}$ is the optimal gap. 'a'= our algorithm, 'b' = primal dual method, 'c'= linearized alm, 'd' = ADMM, 'e' = A-ADMM. Running time counts in seconds. }
	\label{tab4}
	\begin{tabular}{|c|c|c|c|c|c|c|c|c|c|c|c|c|c|}
		\hline  & $\lambda_l$& nnz & opt &\multicolumn{5}{|c|}{ optimal gap, $\eta_{\text{gap}}$} &\multicolumn{5}{|c|}{ running time (iteration number) }   \\
		\hline problem name  & & & & a& b& c& d& e&a&b&c&d&e \\
		$m;n$ &&&&\multicolumn{5}{|c|}{}&\multicolumn{5}{|c|}{}\\
		\hline
		abalone7&$10^{-3}$&23&1.1440+4&1.9-5&9.2+2&9.3+2&6.8+1&6.8+1&\textbf{21($12|96$)}&105(-)&115(-)&250(-)&247(-) \\
		4177;6435&$10^{-4}$&63&9.2897+3&1.3-7&8.6+2&8.7+2&1.8-4&1.8-4&\textbf{30($13|130$)}&104(-)&118(-)&165(6572)&169(6334)\\
		\hline
		bodyfat5&$10^{-3}$&39&5.3609-1&5.0-8&7.4-2&1.4-1&4.5-6&4.5-6&\textbf{3.6($18|94$)}&302(-)&296(-)&534(6892)&544(6837) \\
		252;11628&$10^{-4}$&79&5.5218-2&1.8-8&7.2-2&1.1-1&2.0-6&2.0-6&\textbf{4.3($19|116$)}&296(-)&298(-)&235(2389)&229(2168) \\
		\hline
		housing5&$10^{-3}$&113&2.8392+3&6.1-9&1.4+2&1.4+2&5.8-4&5.8-4&\textbf{5.0($12|106$)}&188(-)&185(-)&371(8558)&365(8368) \\
		506;8568&$10^{-4}$&216&1.0340+3&7.2-9&2.6+2&2.7+2&7.8-5&7.8-5&\textbf{6.7($15|139$)}&179(-)&189(-)&121(2179)&96(1554)\\
		\hline 
		mpg7 & $10^{-3}$&43&1.6769+3&1.1-7&1.2+1&1.9+1&2.1-4&2.1-4&\textbf{1.3($10|78$)}&35(-)&36(-)&19(3755)&18(3484)\\
		392;3432	& $10^{-4}$&132&8.9060+2&8.3-8&4.5+1&5.2+1&4.2-5&4.1-5&\textbf{2.2($14|120$)}&35(-)&40(-)&9.1(1583)&7.5(1014)\\
		\hline
		space\_ga9 & $10^{-3}$&14&3.1931+1&2.7-11&1.1-2&1.3-2&1.9-2&1.9-2&\textbf{6.5($8|48$)}&78(-)&89(-)&161(-)&167(-) \\
		3107;5005 &$10^{-4}$&36&1.9924+1&1.2-8&4.9-1&1.8+0&6.4-6&6.4-6&\textbf{11($11|82$)}&80(-)&97(-)&50(2597)&40(1949) \\
		\hline		
	\end{tabular}
\end{table}
\begin{figure}[!hbp]
	\centering
	\begin{minipage}[!hbp]{0.45\linewidth}
		\includegraphics[width=1.0\textwidth]{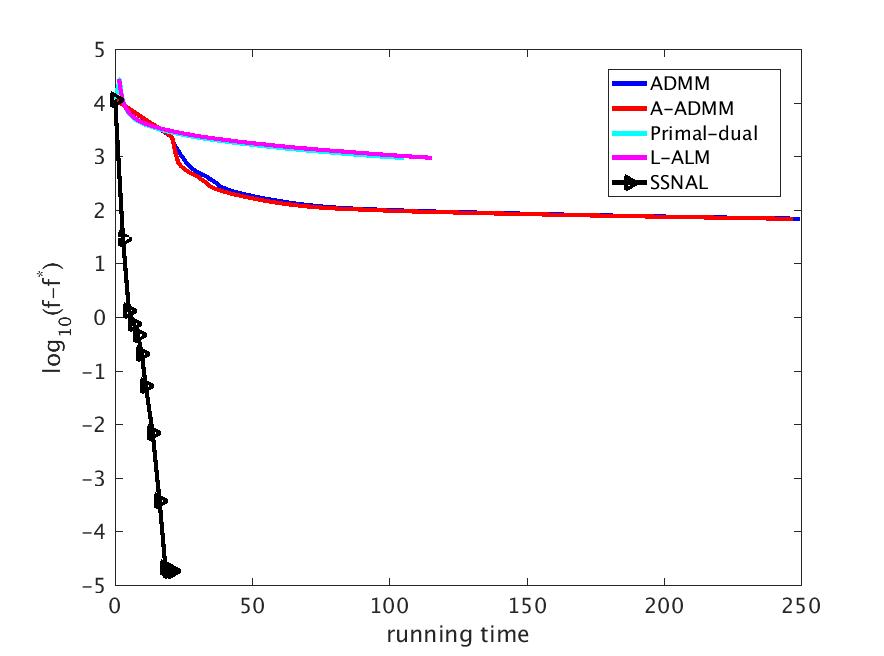}
	\end{minipage}
	\begin{minipage}[!hbp]{0.45\linewidth}
		\includegraphics[width=1.0\textwidth]{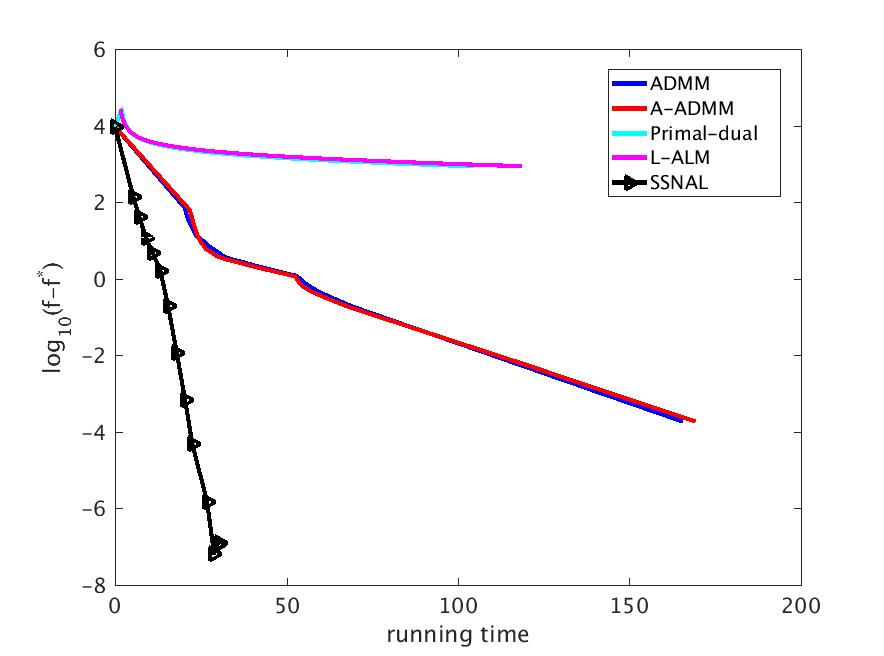}
	\end{minipage}\\
	\begin{minipage}[!hbp]{0.45\linewidth}
		\includegraphics[width=1.0\textwidth]{Figures//bodyfat_scale_expanded5_1_02.jpg}
	\end{minipage} 
	\begin{minipage}[!hbp]{0.45\linewidth}
		\includegraphics[width=1.0\textwidth]{Figures//bodyfat_scale_expanded5_1_0.jpg}
	\end{minipage}
	\caption{Constrained Lasso with sum to zero constraints, $\lambda_l=10^{-3},10^{-4}$. Top two figures are \textbf{abalone7} dataset; bottom two figures are \textbf{bodyfat5} dataset.}	
\end{figure}

\newpage
\subsection{Randomized $B$}
In this case, we generate $B\in\mathbb{R}^{s\times n}$ and $d\in\mathbb{R}^{s}$ randomly and set $s=30$.
\begin{table}[!hbp] 
	\scriptsize
	\caption{\small Performance of SSNAL, primal dual method, linearized ALM, ADMM and A-ADMM with randomized generated constraints on UCI regression data sets. $d$ is the sample size and $n$ is the dimension of each sample. $\lambda_l$ controls the penalty parameter in (P). 'nnz' denotes the number of nonzeros of the solution obtained by our algorithm. 'opt' is the optimal function value of (P). $\eta_{\text{gap}}$ is the optimal gap. 'a'= our algorithm, 'b' = primal dual method, 'c'= linearized alm, 'd' = ADMM, 'e' = A-ADMM. Running time counts in seconds. }
	\label{tab5}
	\begin{tabular}{|c|c|c|c|c|c|c|c|c|c|c|c|c|c|}
		\hline  & $\lambda_l$& nnz & opt &\multicolumn{5}{|c|}{ optimal gap, $\eta_{\text{gap}}$} &\multicolumn{5}{|c|}{ running time (iteration number) }   \\
		\hline problem name  & & & & a& b& c& d& e&a&b&c&d&e \\
		$m;n$ &&&&\multicolumn{5}{|c|}{}&\multicolumn{5}{|c|}{}\\
		\hline
		abalone7&$10^{-3}$& 39&1.2054+4&3.2-7&8.46+2&9.06+2&1.8+1&1.8+1&\textbf{19($8|83$)}&112(-)&163(-)&273(-)&276(-)\\
		4177;6435&$10^{-4}$&70&9.3489+3&1.8-7&8.5+2&9.0+2&2.6-4&2.6-4&\textbf{24($11|98$)}&111(-)&172(-)&149(5298)&142(4957)\\
		\hline
		bodyfat5&$10^{-3}$&49&9.9783-1&4.2-7&9.4-2&5.6-1&2.3-5&2.3-5&\textbf{3.7($12|75$)}&362(-)&394(-)&263(3146)&237(2789) \\
		252;11628&$10^{-4}$&64&1.0369-1&5.2-9&7.4-2&3.3-1&2.3-6&2.2-6&\textbf{4.5($13|97$)}&325(-)&395(-)&123(1294)&98(928) \\
		\hline
		housing5&$10^{-3}$&104&2.8589+3&1.7-8&1.4+2&1.9+2&6.5-4&6.5-4&\textbf{5.1($12|106$)}&200(-)&239(-)&423(9295)&420(9096) \\
		506;8568&$10^{-4}$&218&1.0403+3&2.0-7&2.6+2&3.1+2&8.0-5&8.0-5&\textbf{6.6($17|154$)}&195(-)&252(-)&217(4670)&214(4473)\\
		\hline 
		mpg7 & $10^{-3}$&66&1.8380+3&9.2-7&1.4+1&6.3+1&4.1-4&4.1-4&\textbf{2.0($11|86$)}&38(-)&62(-)&18(3007)&14(1800)\\
		392;3432	& $10^{-4}$&121&8.9682+2&2.8-9&4.5+1&9.4+1&4.2-5&4.1-5&\textbf{2.7($14|121$)}&43(-)&59(-)&13(1745)&8.3(1075)\\
		\hline
		space\_ga9 & $10^{-3}$&49&3.6325+1&3.6-8&2.9-3&1.8-1&1.1-4&1.1-4&\textbf{8.1($9|59$)}&85(-)&123(-)&165(8885)&166(8760) \\
		3107;5005 &$10^{-4}$&74&2.1154+1&5.6-8&4.4-1&4.8+0&9.7-6&8.6-6&\textbf{15($13|101$)}&80(-)&120(-)&24(1074)&21(885)\\
		\hline		
	\end{tabular}
\end{table}
\begin{figure}[!hbp]
	\centering 
	\begin{minipage}[!hbp]{0.45\linewidth}
		\includegraphics[width=1.0\textwidth]{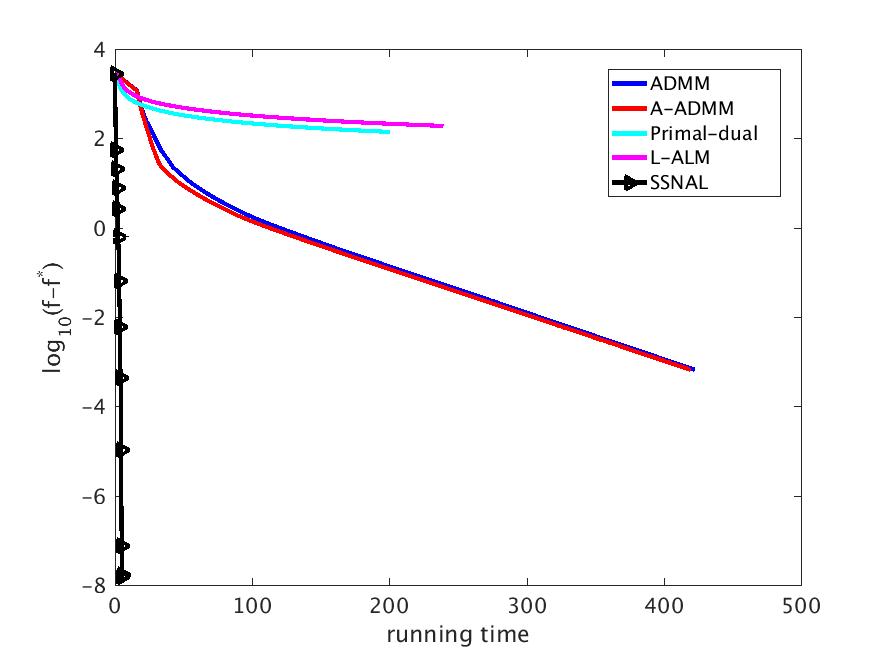}
	\end{minipage}
	\begin{minipage}[!hbp]{0.45\linewidth}
		\includegraphics[width=1.0\textwidth]{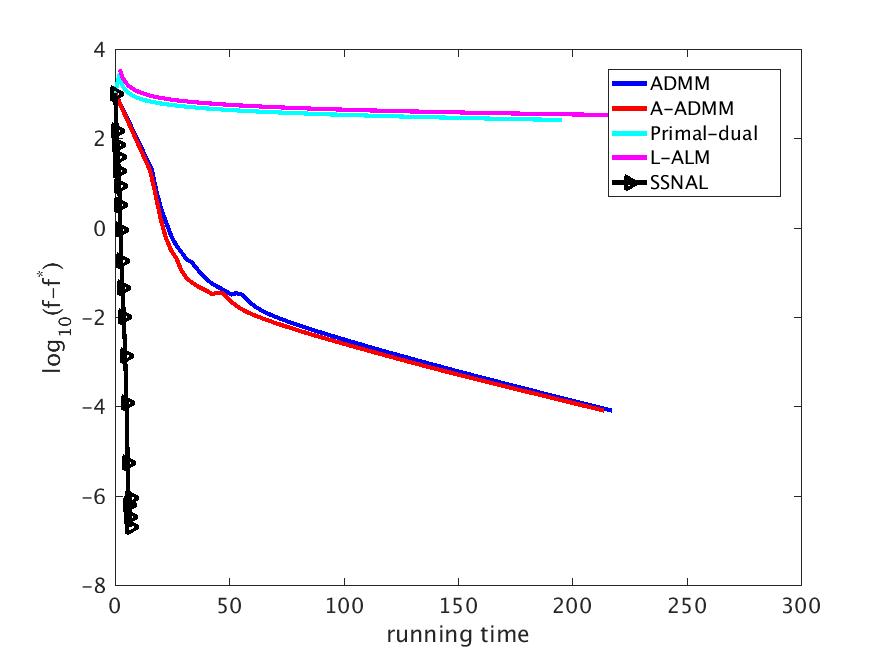}
	\end{minipage} \\
	\begin{minipage}[!hbp]{0.45\linewidth}
		\includegraphics[width=1.0\textwidth]{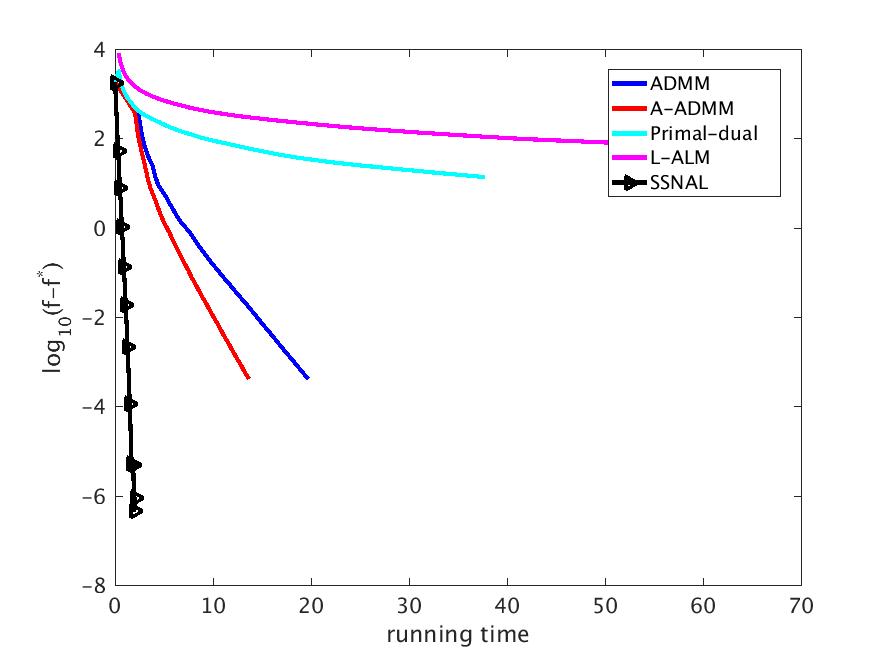}
	\end{minipage}
	\begin{minipage}[!hbp]{0.45\linewidth}
		\includegraphics[width=1.0\textwidth]{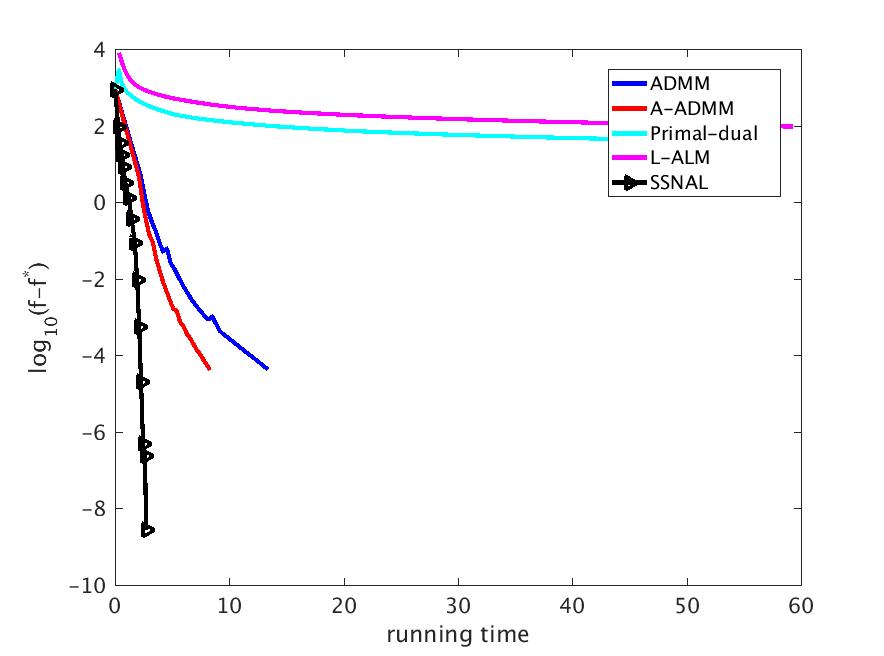}
	\end{minipage}
	\caption{Constrained Lasso with randomized generated constraints, $\lambda_l=10^{-3},10^{-4}$. Top two figures are \textbf{housing5} dataset; bottom two figures are \textbf{mpg7} dataset.}	
\end{figure}

\newpage
\subsection{Generalized Lasso}
In this scenario, we construct $D=\begin{bmatrix}
D_1 \\ D_2
\end{bmatrix}$, where $D_1$ is an $n\times n$ identity matrix and $D_2$ is an $s\times n$ random matrix. Moreover, we set $s=30$.
\begin{table}[!hbp] 
	\scriptsize
	\caption{\small Performance of SSNAL, primal dual method, linearized ALM, ADMM and A-ADMM with generalized Lasso problem on UCI regression data sets. $d$ is the sample size and $n$ is the dimension of each sample. $\lambda_l$ controls the penalty parameter in (P). 'nnz' denotes the number of nonzeros of the solution obtained by our algorithm. 'opt' is the optimal function value of (P). $\eta_{\text{gap}}$ is the optimal gap. 'a'= our algorithm, 'b' = primal dual method, 'c'= linearized alm, 'd' = ADMM, 'e' = A-ADMM. Running time counts in seconds. }
	\label{tab6}
	\begin{tabular}{|c|c|c|c|c|c|c|c|c|c|c|c|c|c|}
		\hline  & $\lambda_l$& nnz & opt &\multicolumn{5}{|c|}{ optimal gap, $\eta_{\text{gap}}$} &\multicolumn{5}{|c|}{ running time (iteration number) }   \\
		\hline problem name  & & & & a& b& c& d& e&a&b&c&d&e \\
		$m;n$ &&&&\multicolumn{5}{|c|}{}&\multicolumn{5}{|c|}{}\\
		\hline
		abalone7&$10^{-3}$&46&1.2013+4&1.1-5&8.2+2&8.2+2&2.4+1&2.4+1&\textbf{51($22|275$)}&127(-)&165(-)&267(-)&270(-) \\
		4177;6435&$10^{-4}$&78&9.3556+3&9.9-7&8.3+2&8.3+2&2.5-4&2.5-4&\textbf{56($12|288$)}&114(-)&176(-)&98(3386)&84(2879)\\
		\hline
		bodyfat5&$10^{-3}$&44&8.2521-1&1.5-7&8.3-2&7.6-2&4.5-5&4.1-5&\textbf{3.7($10|91$)}&321(-)&395(-)&772(-)&783(-) \\
		252;11628&$10^{-4}$&68&8.6951-2&1.9-7&7.2-2&7.2-2&2.3-6&1.7-6&\textbf{4.3($13|105$)}&309(-)&393(-)&121(1244)&93(907) \\
		\hline
		housing5&$10^{-3}$&105&2.8593+3&1.2-5&1.4+2&1.4+2&7.1-4&6.0-4&\textbf{11($12|254$)}&198(-)&246(-)&458(-)&463(9971) \\
		506;8568&$10^{-4}$&219&1.0453+3&2.8-10&2.5+2&2.5+2&8.0-5&8.0-5&7.7($15|178$)&194(-)&257(-)&153(3165)&146(2923)\\
		\hline 
		mpg7 & $10^{-3}$&61&1.8728+3&2.2-6&2.1+1&2.2+1&4.0-4&4.0-4&\textbf{5.9($10|300$)}&40(-)&47(-)&58(9378)&59(9198)\\
		392;3432	& $10^{-4}$&127&8.9665+2&1.7-8&4.6+1&4.6+1&4.2-5&3.7-5&\textbf{3.1($14|146$)}&37(-)&49(-)&9.2(1186)&6.7(767)\\
		\hline
		space\_ga9 & $10^{-3}$&47&3.5451+1&3.7-7&9.1-4&2.6+0&7.4-3&7.3-3&\textbf{15($9|142$)}&83(-)&116(-)&182(-)&180(-) \\
		3107;5005 &$10^{-4}$&78&2.1266+1&2.5-8&4.5-1&4.5-1&9.3-6&1.0-5&\textbf{15($13|125$)}&83(-)&111(-)&30(1431)&24(1132)\\
		\hline		
	\end{tabular}
\end{table}
\begin{figure}[!hbp] 
	\centering
	\begin{minipage}[!hbp]{0.45\linewidth}
		\includegraphics[width=1.0\textwidth]{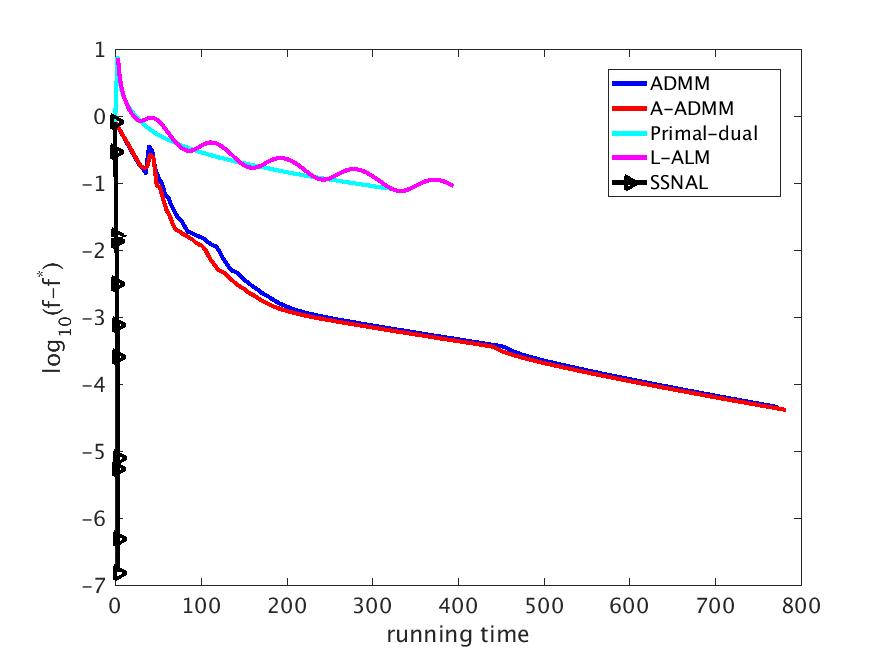}
	\end{minipage}
	\begin{minipage}[!hbp]{0.45\linewidth}
		\includegraphics[width=1.0\textwidth]{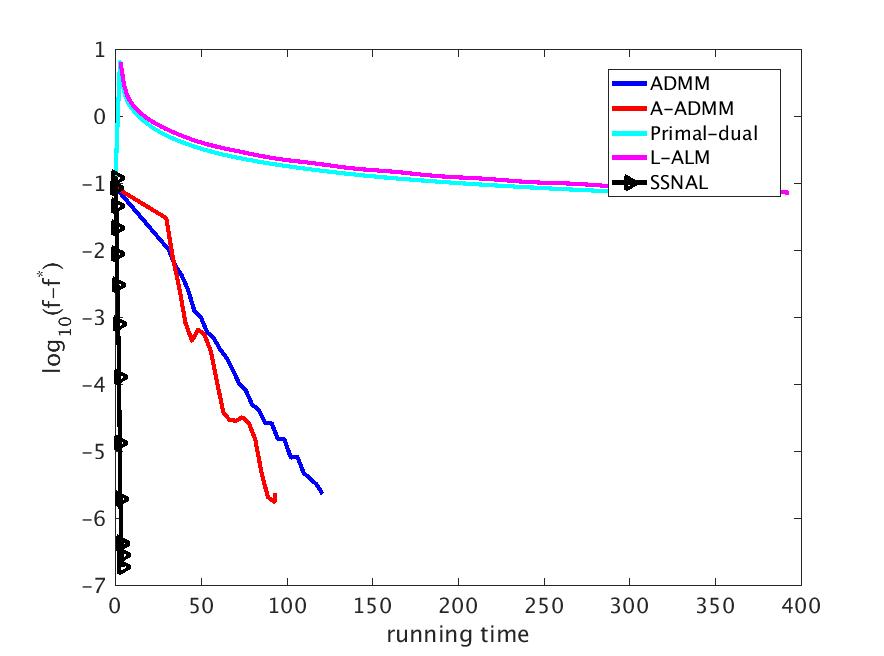}
	\end{minipage} \\
	\begin{minipage}[!hbp]{0.45\linewidth}
		\includegraphics[width=1.0\textwidth]{Figures//mpg_scale_expanded7_30_8.jpg}
	\end{minipage}
	\begin{minipage}[!hbp]{0.45\linewidth}
		\includegraphics[width=1.0\textwidth]{Figures//mpg_scale_expanded7_30_08.jpg}
	\end{minipage}
	\caption{Constrained Lasso with randomized generated constraints, $\lambda_l=10^{-3},10^{-4}$. Top two figures are \textbf{bodyfat5} dataset; bottom two figures are \textbf{mpg7} dataset.}	
\end{figure}

\end{document}